\documentclass{article}

\usepackage{lipsum}
\usepackage{amsfonts}
\usepackage{graphicx}
\usepackage{epstopdf}
\usepackage{algorithmic}

\usepackage{lipsum}
\usepackage{amsfonts}
\usepackage{amsmath}
\usepackage{amsthm}
\usepackage{graphicx}
\usepackage{epstopdf}
\usepackage{algorithmic}
\usepackage{enumerate}
\usepackage{tikz}
\usepackage{hyperref}
\usepackage{cleveref}
\ifpdf
  \DeclareGraphicsExtensions{.eps,.pdf,.png,.jpg}
\else
  \DeclareGraphicsExtensions{.eps}
\fi

\newtheorem{remark}{Remark}
\newtheorem{lemma}{Lemma}
\newtheorem{definition}{Definition}
\newtheorem{corollary}{Corollary}
\newtheorem{proposition}{Proposition}
\newtheorem{theorem}{Theorem}

\crefname{hypothesis}{Hypothesis}{Hypotheses}

\newtheorem{assumption}{Assumption}

\title{Continuous time limit of the stochastic ensemble Kalman inversion: Strong convergence analysis}

\author{Dirk Bl\"omker\thanks{Universit\"at Augsburg, Institut f\"ur Mathematik, 86135 Augsburg, Germany
  ({dirk.bloemker@math.uni-augsburg.de}).}
\and Claudia Schillings\thanks{Universit\"at Mannheim, Institute of Mathematics, D-68131 Mannheim, Germany
  ({c.schillings@uni-mannheim.de}).}
\and Philipp Wacker\thanks{Friedrich-Alexander-Universität Erlangen-Nürnberg, 91058 Erlangen, Germany
  ({phkwacker@gmail.com}).}
\and Simon Weissmann\thanks{Universit\"at Heidelberg, Interdisziplin\"ares Zentrum f\"ur Wissenschaftliches Rechnen, D-69120 Heidelberg, Germany
  ({simon.weissmann@uni-heidelberg.de}).}
  }

\usepackage{amsopn}

\ifpdf
  \DeclareGraphicsExtensions{.eps,.pdf,.png,.jpg}
\else
  \DeclareGraphicsExtensions{.eps}
\fi

\newcommand{\floor}[1]{\left \lfloor #1 \right \rfloor }

\usepackage{color}
\usepackage{todonotes}

\newcommand{\cF}{\mathcal{F}}

\newcommand{\cN}{\mathcal{N}}

\newcommand{\cX}{\mathcal{X}}

\newcommand{\E}{\mathbb{E}}

\newcommand{\N}{\mathbb{N}}

\newcommand{\R}{\mathbb{R}}

\newcommand{\PP}{\mathbb{P}}
\renewcommand{\d}{\mathrm d}
\usepackage{amsopn}
\DeclareMathOperator{\diag}{diag}
\DeclareMathOperator{\HS}{HS}

\newcommand{\trace}{\operatornamewithlimits{Tr}}
\newcommand{\range}{\operatorname{range}}

\usepackage[normalem]{ulem}

\usepackage{subfigure}

\begin{document}

\maketitle

\begin{abstract}
  The Ensemble Kalman inversion (EKI) method is a method for the estimation of unknown parameters in the context of (Bayesian) inverse problems. The method approximates the underlying measure by an ensemble of particles and iteratively applies the ensemble Kalman update to evolve (the approximation of the) prior into the posterior measure.
  For the convergence analysis of the EKI it is common practice to derive a continuous version, replacing the iteration with a stochastic differential equation. In this paper we validate this approach by showing that the stochastic EKI iteration converges to paths of the continuous-time stochastic differential equation by considering both the nonlinear and linear setting, and we prove convergence in probability for the former, and convergence in moments for the latter. The methods employed can also be applied to the analysis of more general numerical schemes for stochastic differential equations in general. 
\end{abstract}



  65N21, 62F15, 65N75, 65C30, 90C56


\section{Introduction}

Inverse problems have a wide range of application in sciences and engineering. The goal is to recover some unknown quantity of interest, which can only be observed indirectly through perturbed observations. These problems are typically ill-posed, in particular solutions often do not depend on the data in a stable way, and regularization techniques are needed in order to overcome the instability. The Bayesian approach to inverse problems interprets the problem in a statistical framework, i.e. introduces a probabilistic model on the parameters and measurements in order to include the underlying uncertainty. The prior distribution on the unknown parameters reflects the prior knowledge on the parameters and regularizes the problem, such that, under suitable assumptions, well-posedness results of the Bayesian problem can be shown. 
The posterior distribution, the solution to the Bayesian inverse problem, is the conditional distribution of the unknown parameters given the observations. Since the posterior distribution is usually not directly accessible, sampling methods for Bayesian inverse problems have become a very active field of research. 

We will focus here on the Ensemble Kalman filter (EnKF) for inverse problems also known as ensemble Kalman inversion (EKI), which is a very popular method for the estimation of unknown parameters in various fields of application. 
Originally, the EnKF has been introduced by Evensen \cite{evensen1994sequential,Evensen2003} for data assimilation problems and more recently, has been considered to solve inverse problems \cite{StLawIg2013}. The EKI has been analysed in the literature as particle approximation of the posterior distribution as well as a derivative-free optimization method for classical inverse problems. Both the EnKF as well as the EKI method have been analyzed in a continuous time formulation formulated by a coupled system of stochastic differential equations (SDEs). The main focus of this work is to theoretically verify the convergence of the discrete EKI method to its continuous time formulation. 

We will give an introduction to our mathematical setup followed by a brief overview of the existing literature.

\subsection{Mathematical setup}

We are interested in solving the inverse problem of recovering the unknown parameter $u\in\cX$ from noisy data $y\in\R^K$ described through the underlying forward model
\begin{align}\label{eq:IP}
y = G(u)+\eta.
\end{align}
Here $G:\cX\to\R^K$ denotes the possibly nonlinear forward map, mapping from a parameter space $\cX$ to an observation space $\R^K$, and $\eta\sim\cN(0,\Gamma)$ models the noise incorporated in the measurement.  Throughout this document we will assume a finite dimensional parameter space $\cX=\R^p$. Due to the subspace property of the EKI, cp. \cite{StLawIg2013}, the EKI ensemble stays in the affine subspace spanned by the initial ensemble, thus rendering the dynamics finite-dimensional. Determinstic approaches to inverse problems typically consider the minimization of a regularized loss functional of the form
\begin{align*}
\min_{u\in\R^p}\ \mathcal L_{\R^K}(G(u),y) + \mathcal R_{\R^p} (u),
\end{align*}
where $\mathcal L_{\R^K}:\R^{K}\times\R^K\to\R_+$ describes the discrepancy of the mapped parameter and the data, whereas $\mathcal R_{\R^p}:\R^p\to\R_+$ is the regularization function incorporating prior information on the parameter $u\in\R^p$.  Classical choices of regularization include Tikhonov regularization \cite{EKN1989} and total variation regularization \cite{Chambolle2009AnIT,ROF1992}.  For more details on the different types of regularization we refer to \cite{EHN1996,benning_burger_2018}.

In contrast, from a statistical point of view, the Bayesian approach for inverse problems incorporates regularization through prior information of the underlying unknown parameter by introducing a probabilistic model.  The unknown parameter $u$ is modeled as an $\R^p$-valued random variable with prior distribution $\mu_0$ which is stochastically independent of the noise $\eta$. Hence, we can view $(u,y)$ as a jointly distributed random variable on $\R^K\times\R^p$ and solving the  Bayesian inverse problem means to condition on the event of the realized observation $y\in\R^K$.  The solution of the Bayesian inverse problem is then given by the distribution of $u\mid y$ also known as the posterior distribution \begin{equation}\label{eq:posterior}
\mu(\mathrm{d}u) = \frac1Z\exp(-\Phi(u;y))\mu_0(\mathrm{d}u),
\end{equation}
with normalization constant 
\[ 
Z:= \int_{\R^p}\exp(-\Phi(u;y))\mu_0(\mathrm{d}u)
\]
and least-squares functional $\Phi(\cdot;y):\R^p\to\R_+$ defined by
\[
\Phi(u;y) = \frac12 \|y-G(u)\|_\Gamma^2,
\]
where $\|\cdot\|_{\Gamma}:=\|\Gamma^{-1/2}\cdot\|$ and $\|\cdot\|$ denotes the euclidean norm in $\R^K$. We note that for a linear forward map $G(\cdot) = A\ \cdot$, $A\in\mathcal L(\R^p,\R^K)$ and Gaussian prior assumption $\mu_0 = \cN(0,\frac1\lambda C_0)$ the maximum a-posteriori estimate computes as
\[
\min_{u\in\R^p}\ \Phi(u;y) + \frac\lambda2\|u\|_{C_0}^2
\]
which relates the Bayesian approach for inverse problems to the Tikhonov regularization with particular choice
\[
\mathcal L_{\R^K}(G(u),y) 
=  \frac12 \|y-G(u)\|_\Gamma^2
\quad \mbox{and}\quad 
\mathcal R_{\R^p} (u)
= \frac\lambda2\|u\|_{C_0}^2.
\]

\subsection{Ensemble Kalman inversion - The ensemble Kalman filter applied to inverse problems}\label{ssec:EKI}
The EKI method,  as it has been originally introduced in \cite{Iglesias2015}, can be viewed as a sequential Monte Carlo method for sampling from the posterior distribution \eqref{eq:posterior}.  The basic idea is to draw an ensemble of samples from the prior distribution and evolve it iteratively through linear Gaussian update steps in order to approximate the posterior distribution. The linear Gaussian update steps are based on the introduced tempered distribution
\begin{equation}\label{eq:tempering}
\mu_{n+1}(\mathrm{d}u) = \frac1{Z_n} \exp(-h\Phi(u;y))\mu_n(\mathrm{d}u),
\end{equation}
with $h=1/N$ and normalizing constants $Z_n$. Note that $\mu_0$ corresponds to the prior distribution and $\mu_N$ to the posterior distribution. 

To make this idea more concrete, we introduce the initial ensemble $(u_0^{(j)})_{j\in\{1,\dots,J\}}$ of size $J$ as an i.i.d.~sample from the prior $u_0^{(j)}\sim\mu_0$. The particle system in the current iteration is used as empirical approximation of the tempering distribution defined in \eqref{eq:tempering}
\[
\mu_n(\mathrm{d}u)\approx \frac1J\sum_{j=1}^J\delta_{u_n^{(j)}}(\mathrm{d}u). 
\]
Given the current particle system $(u_n^{(j)})_{j\in\{1,\dots,J\}}$ we compute the ensemble Kalman filter update for each particle according to obtain a Gaussian approximation on the distribution $\mu_{n+1}$. We define the following empirical means and covariances
\begin{align*}
\bar u_n &= \frac1J\sum_{j=1}^J u_n^{(j)},\quad \bar G_n = \frac1J\sum_{j=1}^J G(u_n^{(j)})\\
C(u_n) &= \frac1J\sum_{j=1}^J (u_n^{(j)}-\bar u_n)(u_n^{(j)}-\bar u_n)^\top,\\
C^{up}(u_n) &= \frac1J\sum_{j=1}^J (u_n^{(j)}-\bar u_n)(G(u_n^{(j)})-\bar G_n)^\top\\
C^{pp}(u_n) &= \frac1J\sum_{j=1}^J (G(u_n^{(j)})-\bar G_n)(G(u_n^{(j)})-\bar G_n)^\top.
\end{align*}

The ensemble Kalman iteration in discrete time is then given by
\begin{align}\label{eq:discrEKI}
u_{n+1}^{(j)} &=u_n^{(j)}-C^{up}(u_n)(C^{pp}(u_n)+h^{-1}\Gamma)^{-1}(G(u_n^{(j)})-y_{n+1}^{(j)}),\quad j=1,\dots,J.
\end{align}
where $h>0$ is the given artificial step size and $y_{n+1}^{(j)}$ are artificially perturbed observation 
\[
y_{n+1}^{(j)} = y+\xi_{n+1}^{(j)},
\]
where $\xi_{n+1}^{(j)}$ are i.i.d.~samples according to $\mathcal N(0,\frac1h \Gamma)$. Considering the EKI iteration in equation \eqref{eq:discrEKI} we find the two parameters $h>0$, denoting the artificial step size, and $J\ge2$, denoting the number of particles. To analyze the EKI method typically at least one of the limits $h\to 0$ or $J\to\infty$ is applied. While the limit $J\to\infty$ is referred to the mean field limit, the limit $h\to0$ corresponds to the continuous time limit of the EKI.  

Our aim is to give a rigorous verification of the continuous time limit for fixed ensemble size $2\le J<\infty.$ Therefore, we first rewrite the discrete EKI formulation \eqref{eq:discrEKI} as
\begin{align*}
u_{n+1}^{(j)} =u_n^{(j)}&-h\,C^{up}(u_n)(h\,C^{pp}(u_n)+\Gamma)^{-1}(G(u_n^{(j)})-y)\\ &+\sqrt{h}\,C^{up}(u_n)(h\,C^{pp}(u_n)+\Gamma)^{-1}\Gamma^{\frac12}\zeta_{n+1}^{(j)},
\end{align*}
where $\zeta_{n+1}^{(j)}$ are i.i.d.~samples according to $\mathcal N(0,E)$. Taking the limit $h\to 0$ leads to $(h\,C^{pp}(u_n)+\Gamma)^{-1}\to \Gamma^{-1}$ and the continuous time limit of the discrete EKI can formally be written as system of coupled stochastic differential equations (SDEs)
\begin{equation}\label{eq:EKI_SDEs}
\mathrm{d} u_t^{(j)} = C^{up}(u_t)\Gamma^{-1}(y-G(u_t^{(j)}))\,\mathrm{d}t + C^{up}(u_t)\Gamma^{-\frac12}\,\mathrm{d}W_t^{(j)},\quad j=1,\dots,J,
\end{equation}
where $W^{(j)} = (W_t^{(j)})_{t\ge0}$ are independent Brownian motions in $\R^p$. We denote by $\widetilde{\mathcal F}_t = \sigma(W_s^{(j)},s\le t)$ the filtration introduced by the Brownian motions and the particle system resulting from the continuous time limit respectively. Furthermore, we denote by $\mathcal F_n = \sigma(\zeta_k^{(j)},\ j=1,\dots,J,\ k\le n)$ the filtration introduced by the increments of the Brownian motion and the particle system resulting from the discrete EKI formulation respectively.  In particular, for the rest of this article we will consider the filtered probability space $(\Omega,\cF,\widetilde\cF=(\widetilde\cF_t)_{t\in[0,T]},\PP)$ and $(\Omega,\cF,\cF=(\cF_n)_{t\in[0,T]},\PP)$ respectively.

We are going to analyze the discrepancy between the discrete EKI formulation and its continuous time limit. Therefore, we introduce a continuous time interpolation of the discrete scheme denoted as $Y(t)$ and we describe the error by $E(t) = Y(t) - u(t)$.  We will provide convergence in probability of the discrete EKI for general nonlinear forward maps, whereas in the linear setting we will provide strong convergence under suitable assumptions.

\subsection{Literature overview}

As stated above the EnKF has been introduced by Evensen \cite{Evensen2003} as a data assimilation method which approximates the filtering distribution based on particles. This method has been first applied in the context of Bayesian inverse problems in \cite{chen2012ensemble,emerick2013ensemble}, and analysed in the large ensemble size limit under linear and Gaussian assumptions \cite{L2009LargeSA, doi:10.1137/140965363} as well as nonlinear models \cite{doi:10.1137/140984415}. In \cite{LS2021_c} the authors study the mean field limit of the closely related ensemble square root filter (ESRF). The EnKF has been formulated in various multilevel formulations \cite{doi:10.1137/15M100955X,2016arXiv160808558C,HST2020,chada2021multilevel}.  A long time and ergodicity analysis are presented in \cite{0951-7715-27-10-2579,47c5c78ebb8c44ef9d8d5e0d23e23c13,0951-7715-29-2-657},  including uniform bounds in time and the incorporation of covariance inflation.  Under linear and Gaussian assumptions the accuracy of the EnKF for a fixed ensemble size has been studied in \cite{Tong2018,doi:10.1002/cpa.21722} as well as the accuracy of the ensemble Kalman-Bucy filter \cite{delmoral2018,doi:10.1137/17M1119056}. Beside the large ensemble size limit, much work has been investigated in the analysis of the continuous time formulation \cite{bergemann2010localization,bergemann2010mollified,Reich2011}.  Theoretical verification of the continuous time limit of the EnKF \cite{LS2021} and the ESRF \cite{LS2021_b} have been derived. In \cite{LS2021}, uniform boundedness on the forward and observation model is assumed. In \cite{lange2021derivation}, this assumption could be relaxed to general nonlinear functions by working with stopping time arguments controlling the the empirical covariances. The results on the continuous time limits then hold locally in time with bounding constants growing exponentially in time.

The application of the EnKF to inverse problems has been proposed in \cite{StLawIg2013}. It can be viewed as a sequential Monte Carlo type method as well as a derivative-free optimization method.  While in the setting of linear forward maps and Gaussian prior assumption the posterior can be approximated in the mean field limit, for nonlinear forward maps this iteration is not consistent with respect to the posterior distribution \cite{ErnstEtAl2015}.  In \cite{DL2021_b,MHGV2018} the authors analyse the mean field limit based on the connection to the Fokker--Planck equation, whereas in \cite{DLL2020} weights have been incorporated in order to correct the resulting posterior estimate for nonlinear models. Much of the existing theory for EKI is based on the continuous time limit resulting in a system of coupled SDEs which has been formally derived in \cite{SchSt2016} and first analysed in \cite{bloemker2018strongly}.  Furthermore, in \cite{Armbruster2020ASO} a stabilized continuous time formulation has been proposed. The continuous time formulation opens up the perspective as a derivativefree optimization method due to its gradient flow structure \cite{SchSt2016,Nikola}.  In the literature two variants are typically considered: the deterministic formulation which basically ignores the diffusion of the underlying SDE and the stochastic formulation including the perturbed observations. In \cite{DBCSPWSW2018} the authors extend the results from \cite{SchSt2016} by showing well-posedness of the stochastic formulation and deriving first convergence results for linear forward models.  The EKI for nonlinear forward models has been studied in \cite{Chada2019ConvergenceAO} in discrete time with nonconstant step size.  In \cite{bungert2021} the dynamical system resulting from the continuous time limit of the EKI has been described and analysed by a spectral decomposition. In the viewpoint of EKI as optimization method it naturally turns out that one has to handle noise in the data.  In \cite{doi:10.1080/00036811.2017.1386784} the authors propose an early stopping criterion based on the Morozov discrepancy and in \cite{Iglesias2015,2016InvPr..32b5002I} discrete regularization has been considered. Most recently, in \cite{CST2020} the authors include Tikhonov regularization within EKI.  Furthermore,  adaptive regularization methods within EKI have been studied in \cite{parzer2021convergence,iglesias2020adaptive}.

In comparison to the EKI method studied in the following, a modified ensemble Kalman sampling method has been introduced in \cite{AGFHWLAS2019} and further analysed in  \cite{GNR2020,DL2021,RW2021}. The basic idea is to shift the noise in the observation to the particle itself and make use of the ergodicity of the resulting SDE related to the Langevin dynamic in order to build a sampling method.

\subsection{Outline of the paper}
The contribution of our document is a rigorous theoretical verification of the continuous time limit of the EKI. We provide two very general results, which can then be applied to the EKI. In particular, we formulate the strong convergence result in a way such that it applies to various variants of the EKI by verifying the existence of moments up to a certain order. We make the following contributions:
\begin{itemize}
\item We present approximation results for a general class of SDEs. Based on localization we are able to bound the error of the discretization up to a stopping time.  Removing the stopping time leads to our two main results: 
\begin{enumerate}
\item convergence in probability with given rate function.
\item convergence in $L^\theta$ with given rate function.
\end{enumerate}
\item We apply the general approximation results to the EKI method in a general nonlinear setting, where we can verify convergence in probability under very weak assumptions on the underlying forward model.
\item In the linear setting we are able to prove strong convergence in $L^\theta$ of the discrete EKI method.  While for general linear forward maps we obtain $L^\theta$ convergence for $\theta\in(0,1)$, we provide various modifications of the scheme in order to ensure $L^\theta$ convergence for $\theta\in(0,2)$.
\end{itemize}

 With this manuscript we resolve the question posed in \cite{bloemker2018strongly}: It is indeed the case that the specific form of the discrete EKI iteration (in particular the additional term $(hC^{pp}(u_n)+\Gamma)^{-1}$ vanishing in the continuous-time limit $h\to 0$) can be thought of as a time-discretization for the SDE \eqref{eq:EKI_SDEs} specifically enforcing strong convergence, which cannot be said for a simple Euler-Maruyama type iteration of form
\begin{align*}
u_{n+1}^{(j)} =u_n^{(j)}&-h\,C^{up}(u_n)(G(u_n^{(j)})-y)\\ &+\sqrt{h}\,C^{up}(u_n)\Gamma^{\frac12}\zeta_{n+1}^{(j)}.
\end{align*}
Indeed, numerical simulations (not presented in this manuscript, but easily implemented) show that the Euler-Maruyama discretization does not exhibit strong convergence (as already demonstrated for a similar SDE in \cite{HJK2011}) due to rare events resulting in exploding iteration paths. There are connections to taming schemes (which have a similar effect of cutting off exploding iterations), as in \cite{hutzenthaler2015numerical}, although the specific form of EKI is not a taming scheme in the narrow sense.

The remainder of this article is structured as follows. In Section~\ref{sec:general_approx} we present our general numerical approximation results for SDEs which are then applied to the solution of general nonlinear inverse problems with the EKI method in Section~\ref{sec:general_EKI}. The application to linear inverse problems is presented in Section~\ref{sec:linear_EKI}.  We close the main part of the document with a brief conclusion in Section~\ref{sec:conclusion} discussing possible further directions to go.
Most of our proofs are shifted to the appendix in order to keep the focus on the key contribution presented in this document.


\section{General approximation results for SDEs}\label{sec:general_approx}

In this section we discuss a general approximation result for SDEs, which is then applied to the ensemble Kalman Inversion. We consider local solutions (i.e., up to a stopping time) of the following general SDE in $\mathbb{R}^n$ in integral notation
\begin{equation}
x(t) = x_0 + \int_0^t f(x(s))dt + \int_0^t g(x(s)) dW(s) , 
\end{equation}
and for $h>0$ we consider the numerical approximations
\begin{equation}\label{eq:cont_interpolation}
Y(t) = x_0+ \int_0^t f_h (Y(\floor{s}))ds + \int_0^t g_h(Y(\floor{s})) dW (s)
\end{equation}
where we round down to the grid 
\[
\floor{s} = \max\{ kh \leq s \ :\ k\in\mathbb{N} \}.
\]
Note that we suppress the index $h$ in the notation. One can check that $Y$ is a continuous time interpolation of the following discrete scheme
\[
Y_{n+1} = Y_n+ h f_h (Y_n) + g_h(Y_n)[W (h(n+1))-W(nh)], \qquad Y_0=x(0).
\]

We assume that both the discrete and the continuous scheme start at the same initial value, i.e.~$x(0)=Y(0)=x_0$. Moreover, for every fixed $h>0$ the discrete scheme exists for all times and cannot blow up in finite time. For the nonlinearities we assume that the limiting drift terms $f$   and the limiting diffusion matrix $g$ are locally Lipschitz and that the nonlinearities $f_h$ and $g_h$ have a uniform bound in $h$ on the growth and approximate $f$ and $g$. To be more precise we formulate the following assumption.

\begin{assumption}
\label{ass:general}
Assume that the functions $f,f_h:\mathbb{R}^p\to \mathbb{R}^p$ and $g,g_h:\mathbb{R}^p\to \mathbb{R}^{p\times m}$,
$h\in(0,1)$ 
are locally Lipschitz such that 
for all radii $R>0$ there exist constants $C_a$, $L$ and $B$ such that for
all $u, v\in\mathbb{R}^n$ with norm less than $R$ the following properties hold:
\begin{enumerate}
\item uniform approximation on compact sets
\[
\|f_h(u)-f(u)\| \leq C_a(R,h) , \qquad \|g_h(u)-g(u)\|_{\HS} \leq C_a(R,h) 
\]
with $C_a(R,h)\to 0$ for $h\to 0$
\item local Lipschitz continuity
\[
\|f(u)-f(v)\| \leq L(R)\|u-v\| , \qquad \|g(u)-g(v)\|_{\HS} \leq L(R)\|u-v\|
\]
\item growth condition
\[
\|f_h(u)\| \leq B(R) , \qquad \|g_h(u) \|_{\HS} \leq B(R).
\]
\end{enumerate}
Moreover, we can assume without loss of generality 
that all $R$-dependent constants are non-decreasing in $R$.
\end{assumption}

\begin{remark}
Note that assumption 3 just means local boundedness, but we will use the specific growth factor $B(R)$ in the proofs later and have to compute the dependence of $B$ on $R$.
To be more precise, we will fix $R$ depending on $h$ such that various terms depending on $h$, $B(R)$, $L(R)$, and $C_a(R,h)$ are small. See for example \eqref{e:condK}.
\end{remark}

Here and in the following we used $\|\cdot\|$ and $ \langle \cdot,\cdot\rangle$ 
for the norm and the standard inner-product in $\mathbb{R}^p$,
while $\| \cdot \|_{\HS}$ is the standard Hilbert-Schmidt norm on matrices in $\mathbb{R}^{p\times p}$
which appears in the Ito formula.

Using the Lipschitz-property, 
we immediately obtain the following statement regarding the one-sided Lipschitz property;
\begin{lemma}
Under Assumption \ref{ass:general} we have for one small $\epsilon>0$ that  
\[
2\langle f(u)- f(y), u-v \rangle + c \|g(u)-g(v)\|_{\HS}^2 \leq (\delta(R)-\epsilon) \|u-v\|^2 
\]
for all $u, v\in\mathbb{R}^n$ with norm less than $R$ with
\begin{equation}\label{eq:def_delta}
\delta(R) :=  2L(R) + c L(R)^2+\epsilon.
\end{equation}
\end{lemma}

This estimate with $c=1+\epsilon$ is needed if we want to bound second moments of the error.
The higher the moment we want to bound, the higher $c$ has to be.

\begin{remark}
The previous lemma provides a weak one-sided Lipschitz property which is 
enough to prove convergence or the error.
Nevertheless, we remark without proof that all the error terms are much smaller, 
if $\delta(R)$ is negative or at least bounded uniformly in $R$. 
We are even able to obtain rates of convergence in that case.

The drawback is that we will need arbitrarily  high moments of a stopped error, which leads to quite technical estimates.
Moreover, determining an optimal $\delta(R)$ is quite delicate in our application we have in mind.
So we postpone these questions to further research.
\end{remark}

Convergence of the Euler-Maruyama scheme for SDEs was postulated under condition of finite exponential moment bounds of the discretization in \cite{HMS2002}, but this condition was soon after proven to be too restrictive: Divergence of the vanilla Euler-Maruyama scheme for non-Lipschitz continuous coefficients was demonstrated in \cite{HJK2011}, due to an exponentially rare (in $h$) family of events with biexponentially bad behavior, which is why standard textbooks about numerical approximations of SDEs \cite{kloeden1992stochastic,lord2014introduction,mao2007stochastic} generally assume globally Lipschitz-continuous coefficients. This led to the development of ``taming schemes'' in \cite{HJK2012,hutzenthaler2015numerical} which are able to cut off the rare tail events leading to exploding moment bounds. The idea is to replace the Euler-Maruyama iteration for an SDE of form $dx = \mu(x)dt + \sigma(x)dW$ of type 
\[ x_{n+1} = x_n + h\cdot \mu(x_n) + \sigma(x_n)\Delta W_n\]
by something of the form
\[x_{n+1} = x_n + \frac{h\cdot \mu(x_n) + \sigma(x_n)\Delta W_n}{1 + |h\cdot \mu(x_n) + \sigma(x_n)\Delta W_n|}.\]
The denominator is close to $1$ for small (well-behaving) increments, and bounds large deviations (which have very small probability anyway) as to avoid exploding paths. Our method of using stopping times to bound (stopped) moments and then remove the stopping times is based on ideas in \cite{HMS2002}.

\subsection{Residual}

We want to bound the error 
\begin{equation}
\label{e:error}
E(t) = x(t)-Y(t)
\end{equation}
solving 
\begin{equation}
dE = [f(x)-f(x+E)]dt +  [g(x)-g(x+E)]dW + d\mathrm{Res}.
\end{equation}
where we define the residual $\mathrm{Res}$, which is an $\mathbb{R}^p$-valued process solving
\begin{equation}
\label{e:Res}
d\mathrm{Res}(t) = [- f_h (Y(\floor{t})) + f(Y(t))] dt +  [-g_h(Y(\floor{t}))+ g(Y(t))] dW.
\end{equation}

Note that the scheme is set up in such a way that $E(0)=0$. Our strategy
of proof is to first bound the error assuming that $E$, $x$ and $Y$ are not too large.
Later we will show that this is true with high probability.

\begin{definition}[cut-off]\label{def:stopping_time}
 For a fixed time $T>0$ and sufficiently large radius $R$ (which will depend on $h$ later) we define the stopping time 
\[
\tau_{R,h} = T \wedge \inf\{t>0: \|x(t)\| > R-1,\ \mathrm{or}\ \|E(t)\|>1 \}.
\]
\end{definition}
Obviously, we have 
\[
\sup_{[0,\tau_{R,h}]}\|x(t)\| \leq R \quad\mbox{and}\quad \sup_{[0,\tau_{R,h}]}\|Y(t)\| \leq R.
\]
Moreover, $\tau_{R,h}>0$ a.s. if $\|x(0)\|<R-1$, as both $x$ and $E$ are  stochastic processes with continuous paths.  We first bound the residual in (\ref{e:Res}):

\begin{lemma}
\label{lem:Res}
For $t\in[0,\tau_{R,h}]$ one has  
\[ d \mathrm{Res}(t) =  \mathrm{Res}_1 (t) dt + \mathrm{Res}_2 (t) dW
\]
with 
\[
\mathbb{E} \sup_{t\in[0,\tau_{R,h}]} \|\mathrm{Res}_1 (t)\|^p \leq C_p  K(R,h)^p
\]
and 
\[
\mathbb{E} \sup_{t\in[0,\tau_{R,h}]} \|\mathrm{Res}_2 (t)\|_{\HS}^p \leq C_p K(R,h)^p
\]
with a constant $C_p>0$ depending only on $p$ and 
\begin{equation}
 \label{e:defK}
 K(R,h):=C_a(R,h)+L(R)h^{1/2} B(R)\;.
\end{equation}
\end{lemma}
As the residual needs to be small in order to prove an approximation result,
in the applications we will need to choose a radius $R=R(h)$, with $R(h)\to \infty$ for $h\to 0$, such that 
\begin{equation}
\label{e:condK}
 K(R(h),h)\to 0 \quad \mbox{for }h\to 0.
\end{equation}

\begin{proof}
For the proof see Appendix~\ref{app:general_approx}.
\end{proof}
%
%
%

\subsection{Moment bound of the Error}

For the error we first prove the following result.
\begin{lemma}
\label{lem:emom}
We have for $K$ from \eqref{e:defK}
\[
\sup_{t\ge0} \mathbb{E}\|E(t \wedge\tau_{R,h})\|^2 \leq 
\left\{
\begin{array}{ccl}
 C K(R,h)^2\cdot \int_0^t e^{ \delta(R) s } ds &&\mbox{ for } \delta(R)>0,\\
C K(R,h)^2 &&\mbox{ for } \delta(R)\leq0.
\end{array}
\right.
\]
\end{lemma}

\begin{proof}[Sketch of the proof]
The main idea here is to apply It\^o's formula in order to derive
\begin{align*}
 d\|E\|^2 &= 2 \langle E, dE \rangle + \langle dE, dE \rangle\nonumber\\
 &= 2 \langle E, [f(x)-f(x+E)] + \mathrm{Res}_1 \rangle dt \nonumber\\
 &\quad + 2 \langle E , [g(x)-g(x+E) + \mathrm{Res}_2] dW \rangle \nonumber\\
 &\quad + \|[g(x)-g(x+E)] +\mathrm{Res}_2 \|_{\HS}^2  dt \nonumber
\end{align*}
and imply 
\begin{align*}
\mathbb{E}\|E(t \wedge\tau_{R,h})\|^2 
&\leq \|E(0)\|^2 + \delta(R) \mathbb{E}\int_0^t\|E(s \wedge\tau_{R,h}) \|^2 dt 
+ C  K(R,h)^2.
\end{align*}
The assertion follows by application of Gronwall's lemma. For full details of the proof see Appendix~\ref{app:general_approx}.
\end{proof}

\begin{remark}
Note that for $\delta(R)\leq C$ (which implies global Lipschitz continuity of $f_h$ and $g_h$ by its definition \eqref{eq:def_delta}), we have a valid error bound as soon as the residuals are small by \eqref{e:condK}.
In the contrast to that in the case $\delta(R)\nearrow\infty$ for $R\to \infty$, 
we might have an additional exponential 
in the bound. Thus we will have to take $R(h)$ much smaller in $h$, 
and we expect it to be some logarithmic term in $h$ at most. 
\end{remark}

We could now proceed and extend this result to  arbitrarily high moments, i.e., we can do estimates  of  $\mathbb{E}\|E(t\wedge\tau_{R,h})\|^p$ by using 
\begin{align*}
d\|E\|^p = d ( \|E\|^2 )^{p/2} &=
p \|E\|^{p-2} \langle E, dE \rangle 
+ \frac{p}2 \|E\|^{p-2} \langle dE, dE \rangle \\ &\quad
+ \frac12 p(p -2) \|E\|^{p-4}\langle E, dE \rangle ^2.
\end{align*}
Each power is now sort of straightforward, 
but needs a different one-sided Lipschitz condition.
To avoid having too many technicalities, we only go up to the 4-th power.
We obtain as before 
\begin{align*}
 d\|E\|^4 &\leq 4 \|E\|^2\langle E, f(x)-f(x+E)+\mathrm{Res}_1 \rangle dt \\ &\quad+ 3 \|E\|^2 \|g(x)-g(x+E)+\mathrm{Res}_2\|_{\HS}^2 dt
 \\ &\quad +2 \|E\|^2 \langle E, [g(x)-g(x+E) + \mathrm{Res}_2] dW  \rangle \\
 &\leq
  \left[ 2 \delta(R) \|E\|^4 +   C_\epsilon \|\mathrm{Res}_2 \|_{\HS}^4 +  C_\epsilon\| \mathrm{Res}_1\|^4 \right] dt 
  \nonumber\\& \quad+  2 \|E\|^2 \langle E, [g(x)-g(x+E) + \mathrm{Res}_2] dW.
\end{align*}

Note that for the fourth power we need a slightly different one-sided Lipschitz condition than for the square. This would yield a different $\delta(R)$. Nevertheless, we slightly abuse notation and consider the same $\delta(R)$, i.e. the larger one, for both cases.
Finally from 
Lemma \ref{lem:Res}, using the martingale property of the stopped integrals, 
\begin{align*}
\mathbb{E}\|E(t \wedge\tau_{R,h})\|^4 
&\leq 2\delta(R) \mathbb{E}\int_0^{t\wedge \tau_{R,h}}\|E\|^4 dt 
+   C_\epsilon T K(R,h)^4 
  \\&\leq 2\delta(R) \mathbb{E}\int_0^t\|E(s \wedge\tau_{R,h}) \|^4 dt 
+ C  K(R,h)^4,
\end{align*}
and again Gronwall's lemma implies:
\begin{lemma}
\label{lem:emom4}
We have for $K$ from \eqref{e:defK}
\[
\sup_{t\ge0} \mathbb{E}\|E(t \wedge\tau_{R,h})\|^4 \leq 
C K(R,h)^4
\left\{
\begin{array}{ccl}
 \int_0^t e^{ 2\delta(R) s } ds &:&\mbox{ for } \delta(R)>0,\\
1 &:&\mbox{ for } \delta(R)\leq0.
\end{array}
\right.
\]
\end{lemma}

\subsection{Uniform moment bound of the error}

With our moment bounds  we now obtain a bound on $\mathbb{E} \sup_{[0,\tau_{R,h}]} \|E\|^2$.

\begin{lemma}\label{lem:uniform_mom}
For all $T>0$ there is a constant $C>0$ such that for $K$ from \eqref{e:defK}
we have
\[
\mathbb{E}\sup_{[0,\tau_{R,h}]}\|E\|^2 
\leq 
C K(R,h)^2(L(R)^2+1)
\left\{
\begin{array}{ccl}
 \int_0^t e^{ 2\delta(R) s } ds &:&\mbox{ for } \delta(R)>0\\
1. &:&\mbox{ for } \delta(R)\leq0
\end{array}
\right.
\]
\end{lemma}

\begin{proof}
For the proof see Appendix~\ref{app:general_approx}.
\end{proof}

Now we can finally fix in applications 
$R(h)\to \infty$ for $h\to0$ (but sufficiently slow) 
such that 
\[ \mathbb{E} \sup_{[0,\tau_{R(h),h}]} \|E\|^2 \to 0 
\mbox{ for } h\to 0.
\]
Let us remark that we could also treat $\mathbb{E} \sup_{[0,\tau_{R,h}]} \|E\|^p$, 
but this will be quite technical and lengthy, using Burkholder-Davis-Gundy.

\subsection{Removing the stopping time}

We present two results depending on how good our bounds are on $x$ and $Y$

\paragraph{Convergence in probability:}
For convergence in probability we only need   stopped moments of $x$, 
as we do not control the error beyond the stopping time.  
Moroeover, these moments can be very weak like logarithmic.

\begin{theorem}\label{thm:general_1}
Assume that there is a radius $R(h)\to\infty$ and a $\gamma(h)\to0$
such that 
\[
\gamma(h)^{-2}\mathbb{E}  \sup_{[0,\tau_{R(h),h}]} \|E\|^2\to0\quad\mathrm{for}\ h\to 0. 
\]
Moreover suppose that for a monotone growing function $\varphi:[0,\infty)\to[0,\infty)$
we have uniformly in $h\in(0,1)$
\[
\mathbb{E}\varphi(x(\tau_{R(h),h})) \leq C.
\]
Then we have 
\[
\mathbb{P} \left(  \sup_{[0,T]} \|E\| > \gamma(h)  \right) \to 0
\quad\mathrm{for}\ h\to 0.
\]
\end{theorem}

\begin{proof}
Consider first using the definition of $\tau_{R,h}$
\begin{align*}
\mathbb{P} \left( \tau_{R,h}<T \right) 
&\leq
\mathbb{P} \left( \|E(\tau_{R,h})\| \geq 1 \mbox{ or } \|x(\tau_{R,h})\| \geq R-1  \right)\\
&\leq \mathbb{P} \left(\|E(\tau_{R,h})\| \geq 1  \right) + \mathbb{P}\left(\|x(\tau_{R,h})\| \geq R-1  \right)\\
&\leq \mathbb{E} \|E(\tau_{R,h})\|^2+ \mathbb{P}\left(\|x(\tau_{R,h})\| \geq R-1  \right)
\end{align*}
Now we obtain 
\begin{align*}
\mathbb{P} \left(  \sup_{[0,T]} \|E\| > \gamma(h)  \right)
&\leq \mathbb{P} \left(  \sup_{[0,T]} \|E\| > \gamma(h) ;\ \tau_{R,h}=T \right)
\\&\quad
+ \mathbb{P} \left(  \sup_{[0,T]} \|E\| > \gamma(h) ;\ \tau_{R,h}<T \right)\\
&\leq 
\mathbb{E}  \sup_{[0,\tau_{R,h}]} \|E\|^2  (1+\gamma(h)^{-2}) 
+ \mathbb{P} (\|x(\tau_{R,h})\|\geq R-1).
\end{align*}
\end{proof}

\paragraph{Convergence in moments:}
In order to bound the moments, we need control of the error beyond the stopping time $\tau_{R,h}$. 
Thus, we need a control on the moments of $x$ and $Y$.
Consider for $\theta>0$ to be fixed later, and $p>1$,
\begin{align*}
 \mathbb{E}\|E(t)\|^\theta 
 &= \int_{\{\tau_{R,h} \geq t \}} \|E(t)\|^\theta d\mathbb{P} + \int_{\{\tau_{R,h} < t \}} \|E(t)\|^\theta d\mathbb{P} \\
 &= \int_{\{\tau_{R,h} \geq t \}} \|E(t\wedge \tau_{R,h})\|^\theta d\mathbb{P}
 + \mathbb{E} \chi_{\{\tau_{R,h} < t \}} \|E(t)\|^\theta d\mathbb{P}\\
 &\leq \mathbb{E} \|E(t\wedge \tau_{R,h})\|^\theta 
 + \mathbb{P}\{\tau_{R,h} < t \}^{(p-1)/p}  \Big(\mathbb{E} \|E(t)\|^{p\theta} \Big)^{1/p} . 
\end{align*}


Now we use first 
\[
  \Big( \mathbb{E}\|E(t)\|^{p\theta} \Big)^{1/p}  
  \leq C \Big( \Big( \mathbb{E}\|x(t)\|^{p\theta} \Big)^{1/p\theta} 
  +  \Big( \mathbb{E}\|Y(t)\|^{p\theta} \Big)^{1/p\theta} \Big)^\theta.
\]
Secondly, we already saw (here $t\in[0,T]$)
 \begin{align*}
 \mathbb{P} \left( \tau_{R,h}< t \right)
&\leq  \mathbb{P} \left( \tau_{R,h}<T \right)\leq  \mathbb{E}  \sup_{[0,\tau_{R,h}]} \|E\|^2 + \mathbb{P} (\|x(\tau_{R,h})\|\geq R-1).
\end{align*}
We obtain the following theorem

\begin{theorem}\label{thm:general_2}
Assume that there is a radius $R(h)\to\infty$ 
such that 
\[
\mathbb{E}  \sup_{[0,\tau_{R(h),h}]} \|E\|^2\to0. 
\]
Moreover suppose that for a monotone growing function $\varphi:[0,\infty)\to[0,\infty)$
we have uniformly in $h\in(0,1)$
\[
\mathbb{E}\varphi(x(\tau_{R(h),h})) \leq C
\]
and suppose the following moment bounds for some $q>0$ 
\[
\sup_{t\in[0,T]}  \mathbb{E}\|x(t)\|^q + \sup_{t\in[0,T]}  \mathbb{E}\|Y(t)\|^q \leq C.
\]
Then we have for any $\theta\in(0,q)\cap(0,2]$ 
\[
\lim_{h\searrow 0} \sup_{t\in[0,T]} \mathbb{E}\|E(t)\|^\theta = 0.
\]
\end{theorem}

\begin{remark}
Let us remark that one can squeeze out a rate of convergence from the proof.
Nevertheless, from the proof one can see that apart from having a $\delta(R)$ 
in the one sided Lipschitz-condition independent of $R$.
For optimal rates we would also need arbitrarily high moments of both $x$ and $y$. See for example \cite{HMS2002}.
\end{remark}

\begin{remark}
Furthermore, we remark without proof that we expect to be able to exchange the $\sup_{t\in[0,T]}$ and expectation in the statements. Actually, many strong convergence results are formulated as $\E\sup_t \|E(t)\|^\theta \to 0$.

For this we anyway have to first prove the result that we stated in the theorem above, and then in a second step improve the estimate by  
using Burkholder inequality. As this would add further technical details and usually halves the order of convergence,
we refrain from giving further details here.
\end{remark}

\section{Application to ensemble Kalman inversion - The nonlinear setting}
\label{sec:general_EKI}

After deriving approximation results for a general class of SDEs, we want to apply the proposed methods in order to quantify the convergence of the discrete EKI algorithm to its continuous version. We start the discussion by recalling our general nonlinear inverse problem
\[ y = G(u) + \eta,
\]
where $u\in\R^p$, $\eta\sim \cN(0,\Gamma)$ for $\Gamma\in\R^{K\times K}$ and $y\in\R^K$. We suppose for simplicity that the forward model $G:\R^p\to\R^K$ is differentiable and grows at most polynomially. To be more precise we assume that there is an $m>0$ and a constant such that for all $u$ 
\begin{equation}
 \label{ass:G}
\|G(u)\| \leq C(1+\|u\|^m) 
\quad\mbox{and}\quad
\|DG(u)\| \leq C(1+\|u\|^{m-1}) 
\end{equation}

Recall that the discrete algorithm of the EKI is given by
\begin{align*}
u_{n+1}^{(j)} =u_n^{(j)}&- h C^{up}(u_n)(h C^{pp}(u_n)+\Gamma)^{-1}(G(u_n^{(j)})- y) \\
&+ h^{1/2} C^{up}(u_n)(h C^{pp}(u_n)+\Gamma)^{-1} \Gamma^{1/2} W_{n+1}^{(j)}),
\end{align*}
while the continuous-time limit is given by the system of coupled SDEs
\begin{equation}\label{e:defcEKI}
\mathrm{d} u_t^{(j)} = C^{up}(u_t)\Gamma^{-1}(y-G(u_t^{(j)}))\,\mathrm{d}t + C^{up}(u_t)\Gamma^{-\frac12}\mathrm{d}W_t^{(j)},
\end{equation}
where the sample covariances are defined in Section~\ref{ssec:EKI}
with ensemble size $J\ge2$ and $W_{n}^{(j)}$ are i.i.d. $\cN(0,1)$ random variables in both $j$ and $n$.
Now consider $u\in \mathbb{R}^{pJ}$ as $u=(u^{(1)},\ldots,u^{(J)})^T$ with $u^{(j)}\in\mathbb{R}^p$ and define the drift $f: \mathbb{R}^{pJ} \to \mathbb{R}^{pJ}$
and the diffusion $g:\mathbb{R}^{pJ} \to \mathbb{R}^{pJ\times pJ}$ by
\[
f^{(j)}(u)= C^{up}(u)\Gamma^{-1}(y-G(u^{(j)}))
\quad\mbox{and}\quad
[g(u) z]_j  = C^{up}(u)\Gamma^{-\frac12} z_j.
\] 
 
The drift and diffusion in the discrete model is given by
\[
f_h^{(j)}(u)= C^{up}(u)(h C^{pp}(u)+\Gamma)^{-1}(G(u^{(j)})- y) 
\]
{and}
\[
[g_h(u) z]_j= C^{up}(u)(h C^{pp}(u)+\Gamma)^{-1} \Gamma^{1/2}z_j,
\] 
while the continuous interpolation $Y$ is defined in \eqref{eq:cont_interpolation} such that $Y(nk)=u_n$.
Consider as before the error $E=u-Y$ between the continous solution $u$ and the continuous interpolation $Y$ of $u_n$.

We first observe that Assumption~\ref{ass:general} is satisfied:
\begin{enumerate}
 \item Obviously, both nonlinear terms are locally Lipschitz, since $G$ is.
 \item The matrix $(hC^{pp}(u) + \Gamma)^{-1}$ is uniformly bounded, such that we have 
$B(R)=C(R^{1+2m}+1)$ in Assumption~\ref{ass:general}.  ($f$ contains $G$ twice)
 \item Similarly, by computing the derivative we obtain that $L(R)=C(R^{2m}+1)$ in Assumption~\ref{ass:general}.
 \item For the approximation, we mainly have to bound 
 \begin{align*}
 \|(hC^{pp}(u) + \Gamma)^{-1}  -  \Gamma^{-1} \|_{\HS} &=  \|h \Gamma^{-1} C^{pp}(u)(hC^{pp}(u) + \Gamma)^{-1}\|_{\HS}\\ &\leq Ch(R^{2m}+1)
 \end{align*}
 which implies that we can choose $C_a(R,h) = Ch(R^{4m+1}+1)$ in Assumption~\ref{ass:general}.
\end{enumerate}

Thus we obtain for $h\in(0,1)$
\[
 K(R,h):= Ch(R^{4m+1}+1) + h^{1/2} C(R^{2m}+1)C(R^{1+2m}+1) \leq Ch^{1/2} (R^{4m+1}+1) .
\]
Moreover, we can choose a trivial bound with 
\[
\delta(R) = C(R^{4m}+1).
\]
Thus, for any fixed $\gamma\in(0,1/2)$,  
we can fix a radius $R(h)\nearrow\infty$ growing very slowly (logarithmically) in $h$ 
such that
using Lemma \ref{lem:emom} (for small $h\to0$)
\begin{align*}
h^{-2\gamma}\sup_{t \geq 0} \mathbb{E}\|E(t\wedge\tau_{R,h})\|^2 
&\leq  Ch^{-2\gamma}K(R(h),h)^2\int_0^Te^{\delta(R(h))} ds \\
&\leq  Ch^{1-2\gamma} R(h)^{8m+2}e^{CR(h)^{4m}} \to 0 \quad \mathrm{for}\ h\to 0.
\end{align*}
We are now ready to rewrite Theorem~\ref{thm:general_1} for the EKI.

\begin{theorem}\label{thm:EKI_general1}
Consider for the EKI with $G$ satisfying (\ref{ass:G}). 
Define the error $E=u-Y$ as above and fix $R(h)\nearrow\infty$ as above. 
Suppose that for a monotone growing function $\varphi:[0,\infty)\to[0,\infty)$ and every $T>0$ 
in the definition of the stopping time $\tau_{R,h}$
we have uniformly for $h\in(0,1)$
\[
\mathbb{E}\varphi(\|u(\tau_{R(h),h})\|) \leq C.
\]
Then for any fixed $\gamma\in(0,1/2)$ and $T>0$
\[
\lim_{h\searrow 0}\mathbb{P}\left(  \sup_{[0,T]} \|E\| > h^{\gamma}\right) =0 \;.
\]
\end{theorem}

Moreover, we can rewrite Theorem~\ref{thm:general_2}.

\begin{theorem}\label{thm:EKI_general2}
Under the setting of Theorem \ref{thm:EKI_general1}
suppose we have for $p>0$ additionally uniform bounds on the 
$p$-th moments of $u$ and $Y$,
i.e. there exists a $C>0$ such that for all $h\in(0,1)$
\[
\sup_{t\in[0,T]}\mathbb{E}\|u(t)\|^p 
+ \sup_{t\in[0,T]}\mathbb{E}\|Y(t)\|^p \le C
\]
then we have for any $\theta\in(0,\min\{2,p\})$ 
\[
\lim_{h\searrow 0} \sup_{[0,T]} \mathbb{E}\|E(t)\|^\theta = 0
\]
\end{theorem}

%

We note that we only need to prove
$\sup_{n\in\{0,\floor{T/h}\}}\mathbb{E}\|u_n\|^2 \leq C $
in the linear case later.
As we have 
\[
Y(t) = \int_0^t f_{h}(Y(\floor{s})) ds +  \int_0^t g_{h}(Y(\floor{s})) dW(s)
\]
with 
\[
 d\|Y(t)\|^2 = 2\langle Y(t), f_{h}(Y\floor{t})\rangle dt +  \|g_{h}(Y\floor{t})\|_{\HS}^2 dt 
 + \langle Y(t), g_{h}(Y\floor{t})dW\rangle
\]
we provide the following interpolation result.

\begin{lemma}[An interpolation lemma for lower moments]\label{lem:interpolation}
Let $u(t) = u_0 + t\cdot f(u_0) + g(u_0)W_t$ with $u_0$, $W_t$ independent and $p\in(0,2)$. Assume further that $\E \|u_0\|^p < C$ and $\E \|u(1)\|^p < C$, then 
\[\E \|u(t)\|^p < C_p \left[ \E\|u_0\|^p + \E\|u(1)\|^p\right]\]
 for all $t\in [0,1]$.
\end{lemma}
\begin{proof}
The proof for this statement is relayed to the appendix.
\end{proof}
We note that we can extend the above result to the whole time interval $[0,T]$ by a shift in time. We leave the details to reader.

In the nonlinear setting based on Theorem \ref{thm:EKI_general1}
we will now prove the following main theorem for globally Lipschitz $G$.
Later in the next section, we will use Theorem \ref{thm:EKI_general2} in the case when $G$ is linear.

\begin{theorem}\label{thm:EKI_main1}
Consider for the EKI with $G$ satisfying (\ref{ass:G}) with $m=1$. Let $u_0 = (u_0^{(j)})_{j\in\{1,\dots,J\}}$ be $\cF_0$-measurable maps $\Omega\to\R^p$ such that $\E[\|u_0^{(j)}\|^2]<\infty$ and suppose $\|y\|\|\Gamma^{-1/2}\|_{\HS}\le C$.
For the error $E=u-Y$ as above 
we have for any fixed $\gamma\in(0,1/2)$ and $T>0$
\[
\lim_{h\searrow 0}\mathbb{P}\left(  \sup_{[0,T]} \|E\| > h^{\gamma}\right) =0 \;.
\]
\end{theorem}

\begin{proof}
For the proof see Appendix~\ref{app:general_EKI}. 
\end{proof}

{
We note that the above result can be used to verify unique strong solutions of the coupled SDEs \eqref{e:defcEKI}. The proposed function $\varphi(\|\bar u\|^2) = \ln(1+\|\bar u\|^2)$ can be used as stochastic Lyapunov function. It is easy to verify that for $V(u) = \varphi(\|\bar u\|^2)$ it holds true that $LV(u) \le C V(u)$ for some constant $C>0$. Thus, by Theorem 3.5 in \cite{Hasminski} we obtain global existence of unique strong solutions.
\begin{corollary}
Under the same assumptions of Theorem~\ref{thm:EKI_main1} for all $T\ge0$ there exists a unique strong solution $(u_t)_{t\in[0,T]}$ (up to $\PP$-indistinguishability) of the set of coupled SDEs \eqref{e:defcEKI}.
\end{corollary}
}

\begin{remark}
We note that assuming that the forward map $G$ takes values $G(u)=0$ for $\|u\|\ge M$, where $M$ is a certain tolerance value, we can directly apply Theorem~\ref{thm:general_2} in order to prove strong convergence of the EKI iteration. This assumption forces the particle system in discrete and continuous time to be bounded and is reasonable if it is known that proper solutions of the underlying inverse problem should be bounded. This assumption can be implemented by modifying the underlying forward map with a smooth shift to $0$ close to the boundary of $\|u\|\in(-M,M)$. The EKI has been analysed under this assumption for example in \cite{Chada2019ConvergenceAO,CST2020}.
\end{remark}

\section{Application to ensemble Kalman inversion - The linear setting}\label{sec:linear_EKI}

We consider the linear inverse problem of recovering an unknown parameter $u\in\R^p$, given noisy observations
\begin{equation}
y = Au + \eta\in\R^K,
\end{equation}
where $\eta\sim \cN(0,\Gamma)$ for $\Gamma\in\R^{K\times K}$. The ensemble Kalman iteration in discrete time is then given by
\begin{align*}
 u_{n+1}^{(j)} &=u_n^{(j)}-C(u_n)A^T(AC(u_n)A^T+h^{-1}\Gamma)^{-1}(Au_n^{(j)}-y_{n+1}^{(j)})\\
				&=u_n^{(j)}-hC(u_n)A^T\Gamma^{-\frac12}(h\Gamma^{-\frac12}AC(u_n)A^T\Gamma^{-\frac12}+I)^{-1}\Gamma^{-\frac12}(Au_n^{(j)}-y_{n+1}^{(j)})
\end{align*}
where we consider perturbed observations $y_{n+1}^{(j)}=y+h^{-\frac12}\Gamma^{\frac12}W_{n+1}^{(j)}$, with $W_{n+1}^{(j)}$ beeing i.i.d. $\cN(0,1)$ random variables and we denote by $\cF_n=\sigma(W_m^{(j)}, m\le n, j=1,\dots,J)$ the filtration introduced by the pertubation. 
Further, we denote the identity matrix $I\in\R^{p}$, we define the scaled forward model $B:=\Gamma^{-\frac12}A$ and write the ensemble Kalman iteration for simplicity as 
\begin{align*}
 u_{n+1}^{(j)} = u_n^{(j)} - hC(u_n) B^T M(u_n)(B u_n^{(j)} - \Gamma^{-\frac12}y) + \sqrt{h} C(u_n) B^TM(u_n)W_{n+1}^{(j)},
\end{align*}
where we have introduced the notation
\begin{equation}
M(u_n) = (hBC(u_n)B^T+I)^{-1}.
\end{equation}

We can decompose $\Gamma^{-\frac12}y = \hat y + \tilde y$, where $\hat y\in \range \Gamma^{-\frac12}A$ and $\tilde y$ is in the orthogonal complement, such that the iteration reads as
\begin{align*}
u_{n+1}^{(j)}&= u_n^{(j)} - hC(u_n) B^T M(u_n)(B u_n^{(j)} - \hat y) +{hC(u_n) B^T M(u_n)\tilde y}\\
&\qquad + \sqrt{h} C(u_n) B^TM(u_n) W_{n+1}^{(j)}.
\end{align*}

Our first result states, that the EKI dynamic ignores the part of observation which takes place in the orthogonal complement of the range of $B$.

\begin{lemma}\label{lem:inner_prod}
Let $\tilde y\in\range(B)^\perp$, then for all $n\in\N$ we have
\[C(u_n) B^T M(u_n)\tilde y = 0.\]
\end{lemma}

\begin{proof}
For the proof see Appendix~\ref{app:linear_EKI}.
\end{proof}

%

Our goal is to apply Theorem~\ref{thm:EKI_general2} in order to prove strong convergence of the ensemble Kalman iteration. To do so, we have to derive bounds on the moments of the continuous time limit $u(t)$ and on the continuous time interpolation of the discrete iteration $Y(t)$. 

We formulate our main result in the following theorem.
\begin{theorem}\label{thm:linear_general}
{Let $u_0 = (u_0^{(j)})_{j\in\{1,\dots,J\}}$ be $\cF_0$-measurable maps $\Omega\to\R^p$ such that $\E[\|u_0^{(j)}\|^{2}]<\infty$.  Furthermore, we assume that the the discrete discrete ensemble Kalman iteration can be bounded uniformly in $h$, i.e.~there exists a $C>0$ such that for all $j\in\{1,\dots,J\}$ it holds true that 
\[
\sup_{n\in\{1,\dots,T\cdot N\}} \ \E[\|u_n^{(j)}\|^p]\le C.
\]
Then we have strong convergence of the approximation error of the EKI method
\[
\lim_{h\searrow 0} \sup_{[0,T]} \mathbb{E}\|E(t)\|^\theta = 0,
\]
for any $\theta\in(0,\min\{2,p\})$}. 
\end{theorem}
\begin{proof}
In order to apply Theorem~\ref{thm:EKI_general2} we have to verify that 
\[ 
\sup_{t\in[0,T]} \E\|u(t)\|^p + \sup_{t\in[0,T]} \E\|Y(t)\|^p
\]
is bounded uniformly in $h$.  Much work has been investigated in the solution of the continuous formulation in \cite{DL2021_b,DBCSPWSW2018}, where $ \sup_{t\in[0,T]} \E\|u(t)\|^p$ can be bounded as the ensemble spread can be bounded in high moments up to $p<J+3$ and hence the bound follows by application of It\^o's formula and Hölder's inequality. Note that this can be seen better in the continuous time formulation
\[ d u_t^{(j)} = \frac1J\sum_{k=1}^J \langle B(u_t^{(j)} - \bar u_t), y-Bu_t^{(j)}+d W_t^{(j)}\rangle (u_t^{(k)}-\bar u_t).
\]
Secondly, we have to bound $\sup_{t\in[0,T]} \E\|Y(t)\|^p$.  We apply the interpolation lemma for the $p$-th moments as the nodes of the interpolation are assumed to be bounded uniformly in $h$ and hence, $\sup_{t\in[0,T]} \E\|Y(t)\|^p\le C$.
\end{proof}

We note that the above result can be used as a general concept in order to prove the strong convergence for different variants of the EKI method as Tikhonov regularized EKI \cite{CST2020}, ensemble Kalman one-shot inversion \cite{GSW2020} or EKI under box-constraints \cite{CSW2019}. Here, the main task is to derive bounds on the discrete ensemble Kalman iteration.  To do so, we present a series of properties which can be used to bound the discrete iteration in moments:
\begin{itemize}
\item We provide a bound on the spread of the particles, i.e.~we prove 
\[\sup_{n\in\{1,\dots,N\}}\E[\|e_n^{(j)}\|^2]< \mathrm{const}.
\]
\item We extend this result by bounding the spread of the particles mapped by $B$, i.e.~we prove \[\sup_{n\in\{1,\dots,N\}}\E[\|B e_n^{(j)}\|^2]< \mathrm{const}.
\]
\item We provide a bound on the residuals mapped by $B$, i.e.~we prove that the data misfit is bounded in the sense that 
\[
\sup_{n\in\{1,\dots,N\}}\E[\|Br_n^{(j)}\|^2]< \mathrm{const}.
\]
\end{itemize}

Using these auxilary results we are then able to provide various strong convergence results under certain assumptions, which are summarized in the following:
\begin{itemize}
\item Our first main result is based on the assumption that the initial ensemble lies outside the kernel of the forward map. While the moments of the dynamical system can be controlled in the image space of $B$, we are not able to control the unobserved part of the system, which is moving in the kernel of $B$. We again obtain strong convergence in the sense that Theorem~\ref{thm:EKI_general2} holds for all $\theta\in(0,2)$.
\item In the second main result we do not state specific assumptions on the forward model, without being linear. However, the strong convergence in Theorem~\ref{thm:EKI_general2} only holds for $\theta\in(0,1)$. 
\item Furthermore, including Tikhonov regularization within EKI we can verify the strong convergence for $\theta\in(0,2)$. 
\end{itemize}

\subsection{Auxiliary result: Bound on the ensemble spread and the residuals}

The update of the ensemble mean is governed by
\begin{equation*}
\bar u_{n+1} = \bar u_n - hC(u_n) B^T M(u_n)(B \bar u_n - \hat y) + \sqrt{h} C(u_n) B^TM(u_n) \bar W_{n+1}
\end{equation*}
with $\bar W_{n+1} = \frac1J\sum_{j=1}^J W_{n+1}^{(j)}$. Further, we set
\[e_n^{(j)} := u_n^{(j)} - \bar u_n,\]
the particle deviation from the mean. Here we get the update formula
\begin{equation*}
e_{n+1}^{(j)} = e_n^{(j)} - hC(u_n) B^T M(u_n)B e_n^{(j)} + \sqrt{h} C(u_n) B^TM(u_n) (W_{n+1}^{(j)} - \bar W_{n+1})
\end{equation*}

We have seen that the update can be written as
\begin{equation}\label{eq:update_particles}
u_{n+1}^{(j)}= u_n^{(j)} - hC(u_n) B^T M(u_n)(B u_n^{(j)} - \hat y) + \sqrt{h} C(u_n) B^TM(u_n)W_{n+1}^{(j)},
\end{equation}
where $\hat y\in\range(B)$, i.e.~there exists $\hat u$, such that $\hat y = B\hat u$. We define the residuals
\begin{equation*}
r_n^{(j)} = u_n^{(j)}-\hat u,
\end{equation*}
where the update of the residuals can be written as
\begin{equation}\label{eq:update_res}
r_{n+1}^{(j)} = r_n^{(j)} - hC(u_n)B^TM(u_n)Br_n^{(j)} + \sqrt{h}C(u_n)B^TM(u_n)W_{n+1}^{(j)}.
\end{equation}

We note that all of the derived auxilary results below crucially depend on the taming through
\begin{equation}
M(u_n) = (hBC(u_n)B^T+I)^{-1},
\end{equation}
suggesting that ignoring $hBC(u_n)B^T$ (which corresponds to an Euler-Maruyama scheme) does not lead to a stable discretization scheme.

Our first useful auxilary result is a bound on the ensemble spread. In particular, we prove that the spread of the particle system is monotonically decreasing in time. This property is very usefull from various perspectives.  First, this property can be used to derive bounds on particle system itself as we can describe the decrease of the spread through a concrete.  Hence, by adding $\frac1J\sum_{j=1}^J  \|e_n^{(j)}\|^2$ to the target value to bound, the increments of the target value decrease.  We will see how to apply this approach in Proposition~\ref{prop:bound_imagespace}.  Second, in the interpretation of EKI as optimization method we are interested in a convergence of the EKI to a point estimate. Hence, we expect each of the particles to converge to the same point.

\begin{lemma}\label{lem:deviationsdecrease}
Let $u_0 = (u_0^{(j)})_{j\in\{1,\dots,J\}}$ be $\cF_0$-measurable maps $\Omega\to\R^p$ such that $\E[\|e_0^{(j)}\|^{2}]<\infty$. Then for all $n\in\N$ it holds true that 
\[
\E\left[\frac1J\sum_{j=1}^J\|e_{n+1}^{(j)}\|^2\right]\le \E\left[\frac1J\sum_{j=1}^J\|e_{n}^{(j)}\|^2\right].
\]
Furthermore, there exists the constant $C=\E[\frac1J\sum_{j=1}^J\|e_0^{(j)}\|^2]$ independent of $h$ such that 
\[
\E\left[\frac1J\sum_{j=1}^J\|e_n^{(j)}\|^2\right]\le C
\]
for all $n\in\N$.
\end{lemma}
\begin{proof}
For the proof see Appendix~\ref{app:linear_EKI}.
\end{proof}

Similarly, the next result states the bound of the particle deviation mapped by $B$.

\begin{corollary}\label{cor:col_obs}
Let $u_0 = (u_0^{(j)})_{j\in\{1,\dots,J\}}$ be $\cF_0$-measurable maps $\Omega\to\R^p$ such that $\E[\|Be_0^{(j)}\|^{2}]<\infty$. Then for all $n\in\N$ it holds true that
\begin{equation*}
\E\left[\frac1J\sum_{j=1}^J\|Be_{n+1}^{(j)}\|^2\right]\le \E\left[\frac1J\sum_{j=1}^J(\|Be_n^{(j)}\|^2\right].
\end{equation*}
\end{corollary}
\begin{proof}
Follows by similar computations as in the proof of Lemma~\ref{lem:deviationsdecrease}.
\end{proof}
For our last auxilary result, we recall that the update of the residuals can be written as
\begin{equation*}
r_{n+1}^{(j)} = r_n^{(j)} - hC(u_n)B^TM(u_n)Br_n^{(j)} + \sqrt{h}C(u_n)B^TM(u_n)W_{n+1}^{(j)}.
\end{equation*}
and provide the boundednes of the residuals in the observation space, which is formulated in the following lemma.
\begin{proposition}\label{prop:bound_imagespace}
For all $n\in\N$ it holds true that
\begin{equation*}
 \sup_{n\in\{1,\dots,T\cdot N\}}\frac1J\sum_{j=1}^J\E[\|Br_{n+1}^{(j)}\|^2+\|Be_{n+1}^{(j)}\|^2]\le \frac1J\sum_{j=1}^J\E[(\|Br_n^{(j)}\|^2+\|Be_n^{(j)}\|^2)].
\end{equation*}
\end{proposition}

\begin{proof}
For the proof see Appendix~\ref{app:linear_EKI}.
\end{proof}

While proving the above two auxilary results, we have derived explicit update formulas for $\frac1J\sum_{j=1}^J\E[\|e_n^{(j)}\|^2]$ and $\frac1J\sum_{j=1}^J\E[\|Br_n^{(j)}\|^2] + \frac1J\sum_{j=1}^J\E[\|Be_n^{(j)}\|^2]$. Using these explicit update formulas, we are further able to bound the following summations.

\begin{corollary}\label{cor:sumbound}
For all $n\in\N$ it holds true that
\begin{align*}
  \frac{J+1}{J}\sum_{k=0}^{n-1}h\E[\|C(u_k)B^\top M(u_k)\|_{\HS}^2] &\le \frac1J\sum_{j=1}^J\E\left[\|e_0^{(j)}\|^2\right],
  \end{align*}
  and
  \begin{align*}
 \sum_{k=0}^{n-1}h\frac1J\sum_{j=1}^J\E[\|C(u_k)^{1/2}B^\top M(u_k)Br_k^{(j)}\|^2]&\le \frac1{2J}\sum_{j=1}^J\E[\|Br_{0}^{(j)}\|^2+\|Be_{0}^{(j)}\|^2].
 \end{align*}
\end{corollary}

\begin{proof}
For the proof see Appendix~\ref{app:linear_EKI}.
\end{proof}


We emphasize that it is not true that the quantity $\frac1J\sum_{j=1}^J\E\|r_{n}^{(j)}\|^2$ is decreasing. This can be seen directly in the continuous and deterministic setting: Here it can be proven that $\frac1J\sum_{j=1}^J \|Br^{(j)}(t)\|^2$ is decreasing, but $\frac1J\sum_{j=1}^J \|r^{(j)}(t)\|^2$ does not have this property.

First, the mapping via $B$ only keeps track of the data-informed parameter dimensions, i.e. those orthogonal to the kernel of $A$. And secondly, even invertibility of $B$ still does not imply monotonicity of $\|\overline{u}(t)-u^\dagger\| $ as the mapping $B$ can warp the coordinate system in such a way that this property is lost. 
This can be seen in an elementary example unrelated to the EKI: Consider the curve $x(t) = (\cos(t),\sin(t))$ for which $V(t) := \|x(t)\|^2$ is constant, i.e. monotonously decreasing. 
On the other hand, with $B=\diag(2,1)$, the mapping $\tilde V(t) = \|Ax(t)\|^2$ is not monotonous.

As a concrete example for the non-monotonicity of the mean and the residual, we can consider the forward operator $A = \diag(100,1)$, observation noise covariance $\Gamma = I_{2\times 2}$, observation $y = (0,0)^T$, and an initial ensemble with mean $\overline{u}_0 = (100,100)^T$ and empirical covariance

\begin{align*}
C(u(0)) =  \left( \begin{array}{cc}
25 & -24\\
-24 & 25\end{array} \right),
\end{align*}

whose eigenvectors are $(-1,1)^T$ and $(1,1)^T$ with eigenvalues $49$ and $1$, respectively.

Figure \ref{fig:nonmonotonicity} shows the initial ensemble and the trajectories of the ensemble and its sample mean in the parameter space.
Clearly, the sample mean and the whole ensemble move away from their final limit $(0,0)^T$ for quite some time until they finally `change direction' and converge towards their limit. 
The initial shearing of the ensemble combined with the strong weighting of the horizontal direction, which is encoded in the forward operator, leads to an initial movement of the ensemble along its principal axis to the top left, increasing the value of $\|\overline r(t)\|^2$.
\begin{figure}[hbt]
    \centering
    {\includegraphics[width=0.6\textwidth]{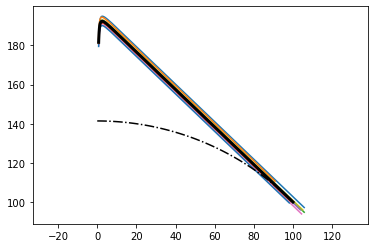}}%
    \caption{Trajectories of the EKI (starting at the lower right corner, black curve is the mean $\overline{u}(t)$) for $t\in[0,1]$. Dotted sphere is the Euclidean sphere through $\overline{u}_0$, demonstrating non-monotonicity of the mean.}
    \label{fig:nonmonotonicity}
\end{figure}

In other words, the Euclidean norm is not the natural norm with respect to which we should view the dynamics of the ensemble and we need to either settle for \textit{non-monotonous convergence} of the residuals $\|\overline{u}(t)-u^\dagger\|$ in parameter space, or we need to pick a more problem-adapted norm. In the deterministic setting, the latter can be done by diagonalizing $C(u(0))B^TB$: It can be shown that this yields a basis of eigenvectors which diagonalize $C(u(t))C(u(0))^{-1}$ for all times, see \cite{bungert2021}. In the stochastic setting, this favourable property is lost.

\subsection{Strong convergence for particle system initialized in the {orthogonal complement of the kernel}}

The key idea of the following proof is to divide the particles dynamics into the dynamics in the kernel of the forward map $B$ and its orthogonal complement. To do so, we introduce the orthogonal projection onto the orthogonal complement of the kernel
\begin{equation*}
P = B^\top(BB^\top)^-B, 
\end{equation*}
where $(BB^\top)^-$ denotes the generalized Moore-Penrose inverse of $BB^\top$. The idea is to split
\[
r_n^{(j)} = Pr_n^{(j)} + (I-P)r_n^{(j)} 
\]
and provide bounds for each term separately.  We can verify bounded second moments of the particle system for the discrete EKI iteration initialized in the image space.
\begin{lemma}\label{lem:secondmoments_imagespace}
Let $u_0 = (u_0^{(j)})_{j\in\{1,\dots,J\}}$ be $\cF_0$-measurable maps $\Omega\to\R^p$ such that $\E[\|Br_0^{(j)}\|^{2}]<\infty$, $ \E[\|(I-P)r_0^{(j)}\|^2]<\infty$ and $(I-P)e_0^{(j)} = 0$ for all $j\in\{1,\dots,J\}$.  Then there exists a uniform constant $C>0$ such that for all $j\in\{1,\dots,J\}$
\begin{equation}
\sup_{n\in\{1,\dots,T\cdot N\}} \E[\|r_n^{(j)}\|^2]\le C.
\end{equation}
\end{lemma}

\begin{proof}
For the proof see Appendix~\ref{app:linear_EKI}.
\end{proof}

\begin{corollary}
Let $u_0 = (u_0^{(j)})_{j\in\{1,\dots,J\}}$ be $\cF_0$-measurable maps $\Omega\to\R^p$ such that $\E[\|u_0^{(j)}\|^2]<\infty$ and $(I-P)e_0^{(j)}=0$ for all $j\in\{1,\dots,J\}$. Then we have strong convergence of the approximation error of the EKI method
\[
\lim_{h\searrow 0} \sup_{[0,T]} \mathbb{E}\|E(t)\|^\theta = 0,
\]
for any $\theta\in(0,2)$
\end{corollary}

\begin{proof}
Direct implication of Lemma~\ref{lem:secondmoments_imagespace} and Theorem~\ref{thm:linear_general}.
\end{proof}
\begin{remark}
We note that the assumption $(I-P)e_0^{(j)}=0$ for all $j\in\{1,\dots,J\}$ could for example be ensured, if the particle system is initialized with $u_0^{(j)} \mapsto P u_0^{(j)}$. However, we mention that through the projection $P$ 	much information about the forward map is necessary, which makes this result quite restrictive.
\end{remark}

\subsection{Strong convergence for general linear forward maps}

{In this section we consider general linear forward models $B=\R^{K\times p}$. While for the previous results we were able to derive second moments of the particle system, in the general setting we derive bounds for $\E[\|u_n\|^\theta]$, for any $\theta\in(0,1)$.
\begin{lemma}\label{lem:1_moments}
There exists a uniform constant $C>0$ such that for all $j\in\{1,\dots,J\}$
\begin{equation}
\sup_{n\in\{1,\dots,T\cdot N\}} \E[\|u_n^{(j)}\|]\le C.
\end{equation}
\end{lemma}
}
\begin{proof}
For the proof see Appendix~\ref{app:linear_EKI}.
\end{proof}

{
\begin{corollary}
Let $u_0 = (u_0^{(j)})_{j\in\{1,\dots,J\}}$ be $\cF_0$-measurable maps $\Omega\to\R^p$ such that $\E[\|u_0^{(j)}\|]<\infty$. Then we have strong convergence of the approximation error of the EKI method
\[
\lim_{h\searrow 0} \sup_{[0,T]} \mathbb{E}\|E(t)\|^\theta = 0,
\]
for any $\theta\in(0,1)$
\end{corollary}
\begin{proof}
Direct implication of Lemma~\ref{lem:1_moments} and Theorem~\ref{thm:linear_general}.
\end{proof}

\begin{remark}
We note that the bound on $\theta<1$ is due to technical reasons and does not come from a fact that there exists no uniform bounds on the moments of the discrete time system. In particular, we expect existence of uniformly bounded moments 
\begin{equation}
\sup_{n\in\{1,\dots,T\cdot N\}} \E[\|u_n^{(j)}\|^p]\le C.
\end{equation}
up to $p=2$ and hence strong convergence up to $\theta<2$. However, for proving bounds in $L^2$ one needs to derive bounds on moments of the ensemble spread in discrete time up to power $4$, which is a challenging task in itself.
\end{remark}

}

\subsection{Strong convergence for Tikhonov regularized EKI and general linear forward maps}

Much of the theoretical based analysis for EKI is based on the viewpoint as optimization method. The analysis is based on the long time behaviour of the scheme, which is the study of the system of coupled SDEs \eqref{e:defcEKI} or the simplified ODE system suppressing the diffusion term for increasing time $T$.  In particular, the aim of EKI in the long time behavior is to solve the minimization problem
\begin{align}\label{eq:ill_posed_opti}
\min_{u}\ \frac12\|G(u)-y\|_{\Gamma}^2
\end{align}
iteratively.  For a linear forward map the motivation behind the EKI as optimization method can be seen by writing the drift term of \eqref{eq:EKI_SDEs} in a preconditioned gradient flow structure
\begin{align*} 
C^{up}(u_t) \Gamma^{-1}(y-Au_t^{(j)}) &= C(u_t)A^\top \Gamma^{-1}(y-Au_t^{(j)})\\ &= -C(u_t)\nabla_u \left(\frac12\|Au^{(j)}-y\|_{\Gamma}^2\right).
\end{align*}
Similarly, in the nonlinear setting, using a second-order approximation,  we can view the drift term of\eqref{eq:EKI_SDEs} as approximation of a preconditioned gradient flow \cite{Nikola},
\begin{align*} 
C^{up}(u_t) \Gamma^{-1}(y-G(u_t^{(j)})) &\approx C(u_t)(DG(u_t^{(j)}))^\top \Gamma^{-1}(y-G(u_t^{(j)}))\\ &= -C(u_t)\nabla_u \left(\frac12\|G(u^{(j)})-y\|_{\Gamma}^2\right).
\end{align*}
Solving the inverse problem through the optimization problem \eqref{eq:ill_posed_opti} is typically ill-posed and regularization is needed.  In \cite{doi:10.1080/00036811.2017.1386784} the authors propose a early stopping criterion based on the Morozov discrepancy principle \cite{Morozov1966OnTS} whereas in \cite{CST2020} Tikhonov regularization has been included into the scheme. We will focus on the Tikhonov regularized ensemble Kalman inversion (TEKI) and prove the strong convergence of the discrete TEKI.  While the TEKI can also be formulated for nonlinear forward maps, we we will focus on the linear setting.

The basic idea of the incorporation of Tikhonov regularization into EKI is to extend the underlying inverse problem \eqref{eq:IP} by prior information. This extension reads as follows
\begin{align*}
y &= Au + \eta,\\
0 & = u + \xi,
\end{align*}
where $\eta\sim\mathcal N(0,\Gamma)$ and $\xi\sim\mathcal N(0,\frac1\lambda C_0)$. Introducing the variables 
\[
\tilde A = \left(\begin{array}{cc}A \\ I \end{array}\right),\quad \tilde y = \left(\begin{array}{c}
y\\ 0\end{array}\right)
,\quad \tilde \eta \sim\mathcal N\left(0,\tilde\Gamma \right),\quad \tilde\Gamma=\left(\begin{array}{cc}
\Gamma & 0 \\ 0 & \frac1\lambda C_0
\end{array}\right)
\]
we can write the extended inverse problem as
\begin{align*}
\tilde y = \tilde A + \tilde\eta.
\end{align*}

For TEKI we now apply EKI to the extended inverse problem which then reads as
\begin{align*}
u_{n+1}^{(j)} =u_n^{(j)}-C(u_n)\tilde A^T(\tilde AC(u_n)\tilde A^T+h^{-1}\tilde\Gamma)^{-1}(\tilde Au_n^{(j)}-\tilde y_{n+1}^{(j)})\\
\end{align*}
with corresponding continuous time limit
\begin{equation}\label{e:defcTEKI}
\mathrm{d} u_t^{(j)} = C(u_t)\tilde A^T\tilde\Gamma^{-1}(\tilde y-\tilde Au_t^{(j)})\,\mathrm{d}t + C(u_t)\tilde A^T\tilde\Gamma^{-\frac12}\mathrm{d}W_t^{(j)},.
\end{equation}
where $W^{(j)}$ are independent Brownian motions in $\R^K\times\mathcal X$. In the long time behavior TEKI can be viewed as optimizer of the regularized objective function
\[
\Phi_R(u,y) = \frac12\|\tilde A u-\tilde y\|_{\R^K\times \mathcal X}^2 = \frac12\|Au-y\|^2 + \frac\lambda2 \|C_0^{-1/2} u\|_{\mathcal X}^2.
\]

The motivation behind this viewpoint can be seen by writing out the drift term of \eqref{e:defcTEKI} as
\begin{align*}
C(u_t)\tilde A^T\tilde\Gamma^{-1}(\tilde y-\tilde Au^{(j)})  &= C(u_t)\left(A^T\Gamma^{-1} (y-Au_t^{(j)}) - \lambda C_0^{-1} u_t^{(j)}\right)\\ &= -C(u_t)\nabla_u \left(\frac12\|A u_t^{(j)} -y\|_\Gamma^2+\frac{\lambda}2\|u_t^{(j)}\|_{C_0}^2\right).
\end{align*}
For a detailed convergence analysis of the TEKI as optimization method we refer to \cite{CST2020}. Since $\tilde A$ and $\tilde B:= \tilde\Gamma^{-1/2}\tilde A$ respectively are linear operators, we can directly apply the above presented results. In particular, we are going to apply Proposition~\ref{prop:bound_imagespace} in order to verify the strong convergence of the discrete TEKI to its continuous time formulation. We prove that the second moments of the particle system are bounded.

\begin{lemma}\label{lem:moments_TEKI}
Let $u_0 = (u_0^{(j)})_{j\in\{1,\dots,J\}}$ be $\cF_0$-measurable maps $\Omega\to\R^p$ such that $\E[\|u_0^{(j)}\|^{2}]<\infty$ for all $j\in\{1,\dots,J\}$.  Then there exists a uniform constant $C>0$ such that for all $j\in\{1,\dots,J\}$
\begin{equation}
\sup_{n\in\{1,\dots,T\cdot N\}} \E[\|u_n^{(j)}\|^2]\le C.
\end{equation}
\end{lemma}

\begin{proof}
For the proof see Appendix~\ref{app:linear_EKI}.
\end{proof}

As we can ensure the bound on the second moments of the particle system we are ready to formulate our main result of strong convergence for the discrete TEKI iteration.
\begin{corollary}
Let $u_0 = (u_0^{(j)})_{j\in\{1,\dots,J\}}$ be $\cF_0$-measurable maps $\Omega\to\R^p$ such that $\E[\|u_0^{(j)}\|^2]<\infty$.  Then we have strong convergence of the approximation error of the TEKI method
\[
\lim_{h\searrow 0} \sup_{[0,T]}\ \mathbb{E}\|E(t)\|^\theta = 0,
\]
for any $\theta\in(0,2)$
\end{corollary}

\section{Conclusion}\label{sec:conclusion}
We have shown that on finite time scales $[0,T]$, the discrete EKI dynamics can be used to approximate the continuous EKI. Or, the other way around, we have established the legitimacy of analyzing the EKI dynamics with a time-continuous model and draw conclusions about the discrete EKI dynamics implemented in practice. For the general nonlinear model, we were able to prove convergence of the discretization in probability, while for the linear setting, we were able to prove even convergence in the $L^\theta$ sense, for $\theta\in (0,1)$, with higher exponents in more favorable settings. We note that the constant derived in the proof still depends on time in the form of $\sqrt T$. Due to the fact that we were able to eliminate dependence on $T$ in the other settings considered (TEKI, convergence in probability for the nonlinear model), we believe that this can be done similarly in the linear setting as well, maybe under additional assumptions, and we leave this as a task for future work.

The methods which we have employed can be used very generally in an SDE setting and can be applied to the analysis of discretization schemes for SDEs in different contexts.

\bibliographystyle{siamplain}
\bibliography{references}

\begin{thebibliography}{10}

\bibitem{Armbruster2020ASO}
{\sc D.~Armbruster, M.~Herty, and G.~Visconti}, {\em A stabilization of a
  continuous limit of the ensemble {K}alman filter}, ArXiv,  (2020),
  \url{https://arxiv.org/abs/2006.15390}.

\bibitem{benning_burger_2018}
{\sc M.~Benning and M.~Burger}, {\em Modern regularization methods for inverse
  problems}, Acta Numerica, 27 (2018), p.~1–111,
  \url{https://doi.org/10.1017/S0962492918000016}.

\bibitem{bergemann2010localization}
{\sc K.~Bergemann and S.~Reich}, {\em A localization technique for ensemble
  {{{Kalman}}} filters}, Quarterly Journal of the Royal Meteorological Society,
  136 (2010), pp.~701--707, \url{https://doi.org/10.1002/qj.591}.

\bibitem{bergemann2010mollified}
{\sc K.~Bergemann and S.~Reich}, {\em A mollified ensemble {{{Kalman}}}
  filter}, Quarterly Journal of the Royal Meteorological Society, 136 (2010),
  pp.~1636--1643, \url{https://doi.org/10.1002/qj.672}.

\bibitem{DBCSPWSW2018}
{\sc D.~Blömker, C.~Schillings, P.~Wacker, and S.~Weissmann}, {\em Well
  posedness and convergence analysis of the ensemble {K}alman inversion},
  Inverse Problems, 35 (2019), p.~085007,
  \url{https://doi.org/10.1088/1361-6420/ab149c}.

\bibitem{bloemker2018strongly}
{\sc D.~Bl{\"o}mker, C.~Schillings, and P.~Wacker}, {\em A strongly convergent
  numerical scheme from ensemble {K}alman inversion}, SIAM Journal on Numerical
  Analysis, 56 (2018), pp.~2537--2562,
  \url{https://doi.org/10.1137/17M1132367}.

\bibitem{bungert2021}
{\sc L.~Bungert and P.~Wacker}, {\em {Long-time behaviour and spectral
  decomposition of the linear ensemble {K}alman inversion in parameter space}},
  ArXiv,  (2021), \url{https://arxiv.org/abs/2104.13281}.

\bibitem{chada2021multilevel}
{\sc N.~K. Chada, A.~Jasra, and F.~Yu}, {\em Multilevel ensemble {K}alman-bucy
  filters}, 2021, \url{https://arxiv.org/abs/2011.04342}.

\bibitem{CSW2019}
{\sc N.~K. Chada, C.~Schillings, and S.~Weissmann}, {\em On the incorporation
  of box-constraints for ensemble {K}alman inversion}, Foundations of Data
  Science, 1 (2019), p.~433, \url{https://doi.org/10.3934/fods.2019018}.

\bibitem{CST2020}
{\sc N.~K. Chada, A.~M. Stuart, and X.~T. Tong}, {\em Tikhonov regularization
  within ensemble {K}alman inversion}, SIAM Journal on Numerical Analysis, 58
  (2020), pp.~1263--1294, \url{https://doi.org/10.1137/19M1242331}.

\bibitem{Chada2019ConvergenceAO}
{\sc N.~K. Chada and X.~T. Tong}, {\em Convergence acceleration of ensemble
  {K}alman inversion in nonlinear settings}, ArXiv e-prints,  (2019),
  \url{https://arxiv.org/abs/1911.02424}.

\bibitem{Chambolle2009AnIT}
{\sc A.~Chambolle, V.~Caselles, D.~Cremers, M.~Novaga, and T.~Pock}, {\em An
  Introduction to Total Variation for Image Analysis}, De Gruyter, Berlin,
  Boston, 16 Jul. 2010,
  \url{https://doi.org/https://doi.org/10.1515/9783110226157.263}.

\bibitem{chen2012ensemble}
{\sc Y.~Chen and D.~S. Oliver}, {\em Ensemble randomized maximum likelihood
  method as an iterative ensemble smoother}, Mathematical Geosciences, 44
  (2012), pp.~1--26, \url{https://doi.org/10.1007/s11004-011-9376-z}.

\bibitem{2016arXiv160808558C}
{\sc A.~{Chernov}, H.~{Hoel}, K.~{Law}, F.~{Nobile}, and R.~{Tempone}}, {\em
  {Multilevel ensemble {Kalman} filtering for spatially extended models}},
  ArXiv e-prints,  (2016), \url{https://arxiv.org/abs/1608.08558}.

\bibitem{doi:10.1137/17M1119056}
{\sc J.~de~Wiljes, S.~Reich, and W.~Stannat}, {\em Long-time stability and
  accuracy of the ensemble {Kalman}--bucy filter for fully observed processes
  and small measurement noise}, SIAM Journal on Applied Dynamical Systems, 17
  (2018), pp.~1152--1181, \url{https://doi.org/10.1137/17M1119056}.

\bibitem{delmoral2018}
{\sc P.~Del~Moral and J.~Tugaut}, {\em On the stability and the uniform
  propagation of chaos properties of ensemble {K}alman bucy filters}, The
  Annals of Applied Probability, 28 (2018), pp.~790--850,
  \url{https://doi.org/10.1214/17-AAP1317}.

\bibitem{DL2021_b}
{\sc Z.~Ding and Q.~Li}, {\em Ensemble {K}alman inversion: mean-field limit and
  convergence analysis}, Statistics and Computing, 31 (2021), p.~9,
  \url{https://doi.org/10.1007/s11222-020-09976-0}.

\bibitem{DL2021}
{\sc Z.~Ding and Q.~Li}, {\em Ensemble {K}alman sampler: Mean-field limit and
  convergence analysis}, SIAM Journal on Mathematical Analysis, 53 (2021),
  pp.~1546--1578, \url{https://doi.org/10.1137/20M1339507}.

\bibitem{DLL2020}
{\sc Z.~Ding, Q.~Li, and J.~Lu}, {\em Ensemble {K}alman inversion for nonlinear
  problems: Weights, consistency, and variance bounds}, Foundations of Data
  Science, 0 (2020), pp.~--, \url{https://doi.org/10.3934/fods.2020018}.

\bibitem{emerick2013ensemble}
{\sc A.~A. Emerick and A.~C. Reynolds}, {\em Ensemble smoother with multiple
  data assimilation}, Computers \& Geosciences, 55 (2013), pp.~3--15.

\bibitem{EHN1996}
{\sc H.~Engl, M.~Hanke, and G.~Neubauer}, {\em Regularization of Inverse
  Problems}, Mathematics and Its Applications, Springer Netherlands, 1996,
  \url{https://books.google.de/books?id=DF7R\_fVLuM8C}.

\bibitem{EKN1989}
{\sc H.~W. Engl, K.~Kunisch, and A.~Neubauer}, {\em Convergence rates for
  {T}ikhonov regularisation of non-linear ill-posed problems}, Inverse
  Problems, 5 (1989), pp.~523--540,
  \url{https://doi.org/10.1088/0266-5611/5/4/007}.

\bibitem{ErnstEtAl2015}
{\sc O.~G. Ernst, B.~Sprungk, and H.-J. Starkloff}, {\em Analysis of the
  ensemble and polynomial chaos {K}alman filters in {B}ayesian inverse
  problems}, SIAM/ASA Journal on Uncertainty Quantification, 3 (2015),
  pp.~823--851, \url{https://doi.org/10.1137/140981319}.

\bibitem{evensen1994sequential}
{\sc G.~Evensen}, {\em Sequential data assimilation with a nonlinear
  quasi-geostrophic model using {M}onte {C}arlo methods to forecast error
  statistics}, Journal of Geophysical Research: Oceans, 99 (1994),
  pp.~10143--10162, \url{https://doi.org/10.1029/94JC00572}.

\bibitem{Evensen2003}
{\sc G.~Evensen}, {\em The {E}nsemble {K}alman filter: theoretical formulation
  and practical implementation}, Ocean Dynamics, 53 (2003), pp.~343--367,
  \url{https://doi.org/10.1007/s10236-003-0036-9}.

\bibitem{AGFHWLAS2019}
{\sc A.~Garbuno-Inigo, F.~Hoffmann, W.~Li, and A.~M. Stuart}, {\em Interacting
  {L}angevin diffusions: Gradient structure and ensemble {K}alman sampler},
  SIAM Journal on Applied Dynamical Systems, 19 (2020), pp.~412--441,
  \url{https://doi.org/10.1137/19M1251655}.

\bibitem{GNR2020}
{\sc A.~Garbuno-Inigo, N.~Nüsken, and S.~Reich}, {\em Affine invariant
  interacting {L}angevin dynamics for {B}ayesian inference}, SIAM Journal on
  Applied Dynamical Systems, 19 (2020), pp.~1633--1658,
  \url{https://doi.org/10.1137/19M1304891}.

\bibitem{GSW2020}
{\sc P.~A. Guth, C.~Schillings, and S.~Weissmann}, {\em Ensemble {K}alman
  filter for neural network based one-shot inversion}, 2020,
  \url{https://arxiv.org/abs/2005.02039}.

\bibitem{MHGV2018}
{\sc M.~Herty and G.~Visconti}, {\em Kinetic methods for inverse problems},
  Kinetic \& Related Models, 12 (2019), p.~1109,
  \url{https://doi.org/10.3934/krm.2019042}.

\bibitem{HMS2002}
{\sc D.~J. Higham, X.~Mao, and A.~M. Stuart}, {\em Strong convergence of
  {E}uler-type methods for nonlinear stochastic differential equations}, SIAM
  Journal on Numerical Analysis, 40 (2002), pp.~1041--1063,
  \url{https://doi.org/10.1137/S0036142901389530}.

\bibitem{doi:10.1137/15M100955X}
{\sc H.~Hoel, K.~Law, and R.~Tempone}, {\em Multilevel ensemble {Kalman}
  filtering}, SIAM Journal on Numerical Analysis, 54 (2016), pp.~1813--1839,
  \url{https://doi.org/10.1137/15M100955X}.

\bibitem{HST2020}
{\sc H.~Hoel, G.~Shaimerdenova, and R.~Tempone}, {\em Multilevel ensemble
  {K}alman filtering based on a sample average of independent {EnKF}
  estimators}, Foundations of Data Science, 2 (2020), pp.~351--390,
  \url{https://doi.org/10.3934/fods.2020017}.

\bibitem{hutzenthaler2015numerical}
{\sc M.~Hutzenthaler and A.~Jentzen}, {\em Numerical approximations of
  stochastic differential equations with non-globally Lipschitz continuous
  coefficients}, American Mathematical Soc., 2015.

\bibitem{HJK2011}
{\sc M.~Hutzenthaler, A.~Jentzen, and P.~E. Kloeden}, {\em Strong and weak
  divergence in finite time of {E}uler's method for stochastic differential
  equations with non-globally {L}ipschitz continuous coefficients}, Proceedings
  of the Royal Society A: Mathematical, Physical and Engineering Sciences, 467
  (2011), pp.~1563--1576, \url{https://doi.org/10.1098/rspa.2010.0348}.

\bibitem{HJK2012}
{\sc M.~Hutzenthaler, A.~Jentzen, and P.~E. Kloeden}, {\em Strong convergence
  of an explicit numerical method for sdes with nonglobally {L}ipschitz
  continuous coefficients}, The Annals of Applied Probability, 22 (2012),
  pp.~1611--1641, \url{http://www.jstor.org/stable/41713370}.

\bibitem{iglesias2020adaptive}
{\sc M.~Iglesias and Y.~Yang}, {\em Adaptive regularisation for ensemble
  {K}alman inversion}, 2020, \url{https://arxiv.org/abs/2006.14980}.

\bibitem{Iglesias2015}
{\sc M.~A. Iglesias}, {\em Iterative regularization for ensemble data
  assimilation in reservoir models}, Computational Geosciences, 19 (2015),
  pp.~177--212, \url{https://doi.org/10.1007/s10596-014-9456-5}.

\bibitem{2016InvPr..32b5002I}
{\sc M.~A. Iglesias}, {\em A regularizing iterative ensemble {K}alman method
  for {PDE}-constrained inverse problems}, Inverse Problems, 32 (2016),
  p.~025002, \url{http://stacks.iop.org/0266-5611/32/i=2/a=025002}.

\bibitem{StLawIg2013}
{\sc M.~A. Iglesias, K.~Law, and A.~M. Stuart}, {\em Ensemble {K}alman methods
  for inverse problems}, Inverse Problems, 29 (2013), p.~045001,
  \url{http://stacks.iop.org/0266-5611/29/i=4/a=045001}.

\bibitem{0951-7715-27-10-2579}
{\sc D.~Kelly, K.~Law, and A.~M. Stuart}, {\em Well-posedness and accuracy of
  the ensemble {K}alman filter in discrete and continuous time}, Nonlinearity,
  27 (2014), p.~2579, \url{http://stacks.iop.org/0951-7715/27/i=10/a=2579}.

\bibitem{0951-7715-29-2-657}
{\sc D.~Kelly, A.~J. Majda, and X.~T. Tong}, {\em Nonlinear stability and
  ergodicity of ensemble based {K}alman filters}, Nonlinearity, 29 (2016),
  p.~657, \url{http://stacks.iop.org/0951-7715/29/i=2/a=657}.

\bibitem{Hasminski}
{\sc R.~Z. Khasminskii}, {\em Stochastic stability of differential equations.
  Transl. by D. Louvish. Ed. by S. Swierczkowski.}, Monographs and Textbooks on
  Mechanics of Solids and Fluids. Mechanics: Analysis, 7. Alphen aan den Rijn,
  The Netherlands; Rockville, Maryland, USA, Sijthoff \& Noordhoff, 1980.

\bibitem{kloeden1992stochastic}
{\sc P.~E. Kloeden and E.~Platen}, {\em Numerical Solution of Stochastic
  Differential Equations}, Springer, 1992.

\bibitem{Nikola}
{\sc N.~B. Kovachki and A.~M. Stuart}, {\em Ensemble {K}alman inversion: a
  derivative-free technique for machine learning tasks}, Inverse Problems, 35
  (2019), p.~095005, \url{https://doi.org/10.1088/1361-6420/ab1c3a}.

\bibitem{doi:10.1137/140965363}
{\sc E.~Kwiatkowski and J.~Mandel}, {\em Convergence of the square root
  ensemble {Kalman} filter in the large ensemble limit}, SIAM/ASA Journal on
  Uncertainty Quantification, 3 (2015), pp.~1--17,
  \url{https://doi.org/10.1137/140965363}.

\bibitem{lange2021derivation}
{\sc T.~Lange}, {\em Derivation of ensemble {K}alman-bucy filters with
  unbounded nonlinear coefficients}, 2021,
  \url{https://arxiv.org/abs/2012.07572}.

\bibitem{LS2021_c}
{\sc T.~Lange and W.~Stannat}, {\em Mean field limit of ensemble square root
  filters - discrete and continuous time}, Foundations of Data Science, 0
  (2021), pp.~--, \url{https://doi.org/10.3934/fods.2021003}.

\bibitem{LS2021_b}
{\sc T.~Lange and W.~Stannat}, {\em On the continuous time limit of ensemble
  square root filters}, 2021, \url{https://arxiv.org/abs/1910.12493}.

\bibitem{LS2021}
{\sc T.~Lange and W.~Stannat}, {\em On the continuous time limit of the
  ensemble {K}alman filter}, Math. Comput., 90 (2021), pp.~233--265,
  \url{https://doi.org/10.1090/mcom/3588}.

\bibitem{doi:10.1137/140984415}
{\sc K.~Law, H.~Tembine, and R.~Tempone}, {\em Deterministic mean-field
  ensemble {Kalman} filtering}, SIAM Journal on Scientific Computing, 38
  (2016), pp.~A1251--A1279, \url{https://doi.org/10.1137/140984415}.

\bibitem{L2009LargeSA}
{\sc F.~Le~Gland, V.~Monbet, and V.-D. Tran}, {\em {Large sample asymptotics
  for the ensemble {{{Kalman}}} filter}}, Research Report RR-7014, {INRIA},
  2009, \url{https://hal.inria.fr/inria-00409060}.

\bibitem{lord2014introduction}
{\sc G.~J. Lord, C.~E. Powell, and T.~Shardlow}, {\em An introduction to
  computational stochastic PDEs}, vol.~50, Cambridge University Press, 2014.

\bibitem{doi:10.1002/cpa.21722}
{\sc A.~J. Majda and X.~T. Tong}, {\em Performance of ensemble {K}alman filters
  in large dimensions}, Communications on Pure and Applied Mathematics, 71
  (2018), pp.~892--937, \url{https://doi.org/10.1002/cpa.21722}.

\bibitem{mao2007stochastic}
{\sc X.~Mao}, {\em Stochastic differential equations and applications},
  Elsevier, 2007.

\bibitem{Morozov1966OnTS}
{\sc V.~A. Morozov}, {\em On the solution of functional equations by the method
  of regularization}, Dokl. Akad. Nauk SSSR, 167 (1966), pp.~510--512.

\bibitem{parzer2021convergence}
{\sc F.~Parzer and O.~Scherzer}, {\em On convergence rates of adaptive ensemble
  {K}alman inversion for linear ill-posed problems}, 2021,
  \url{https://arxiv.org/abs/2104.10895}.

\bibitem{Reich2011}
{\sc S.~Reich}, {\em A dynamical systems framework for intermittent data
  assimilation}, BIT Numerical Mathematics, 51 (2011), pp.~235--249,
  \url{https://doi.org/10.1007/s10543-010-0302-4}.

\bibitem{RW2021}
{\sc S.~Reich and S.~Weissmann}, {\em {F}okker--{P}lanck particle systems for
  {B}ayesian inference: Computational approaches}, SIAM/ASA Journal on
  Uncertainty Quantification, 9 (2021), pp.~446--482,
  \url{https://doi.org/10.1137/19M1303162}.

\bibitem{ROF1992}
{\sc L.~I. Rudin, S.~Osher, and E.~Fatemi}, {\em Nonlinear total variation
  based noise removal algorithms}, Physica D: Nonlinear Phenomena, 60 (1992),
  pp.~259 -- 268, \url{https://doi.org/10.1016/0167-2789(92)90242-F}.

\bibitem{SchSt2016}
{\sc C.~Schillings and A.~M. Stuart}, {\em Analysis of the ensemble {K}alman
  filter for inverse problems}, SIAM Journal on Numerical Analysis, 55 (2017),
  pp.~1264--1290, \url{https://doi.org/10.1137/16M105959X}.

\bibitem{doi:10.1080/00036811.2017.1386784}
{\sc C.~Schillings and A.~M. Stuart}, {\em Convergence analysis of ensemble
  {Kalman} inversion: the linear, noisy case}, Applicable Analysis, 97 (2018),
  pp.~107--123, \url{https://doi.org/10.1080/00036811.2017.1386784}.

\bibitem{47c5c78ebb8c44ef9d8d5e0d23e23c13}
{\sc X.~Tong, A.~Majda, and D.~Kelly}, {\em Nonlinear stability of the ensemble
  {K}alman filter with adaptive covariance inflation}, Communications in
  Mathematical Sciences, 14 (2016), pp.~1283--1313,
  \url{https://doi.org/10.4310/CMS.2016.v14.n5.a5}.

\bibitem{Tong2018}
{\sc X.~T. Tong}, {\em Performance analysis of local ensemble {Kalman} filter},
  Journal of Nonlinear Science, 28 (2018), pp.~1397--1442,
  \url{https://doi.org/10.1007/s00332-018-9453-2}.

\end{thebibliography}

\appendix

\section{Proofs of Section~\ref{sec:general_approx}}\label{app:general_approx}

\begin{proof}[Proof of Lemma~\ref{lem:Res}]
We start by bounding the error between $Y(t)$ for $t\in[kh,(k+1)]h$ 
and $Y_k$.  By the SDE for the approximation 
\begin{align*}
\|Y(\floor{t})-Y(t)\| 
&=   \|\int_{\floor{t}}^t  f_h (Y(\floor{s}))ds + \int_{\floor{t}}^t g_h(Y(\floor{s})) dW(s)\|  \\
&\leq h B(R) + B(R) \|W(t)-W(\floor{s})\|
\end{align*}
Thus, by the Burkholder--Davis--Gundy inequality and by merging the higher-order term $h^p$ into the lower-order term $h^{p/2}$ with an appropriate constant,
\[ \mathbb{E} \sup_{[0,\tau_{R,h}]} \|Y(\floor{t})-Y(t)\|^p \leq C_p h^{p/2} B(R)^p.
\]

In order to bound the residual, we consider $t\in[0,\tau_{R,h}]$
and thus $\floor{t}\in[0,\tau_{R,h}]$ with   $\|Y(\floor{t})\|\leq R$ . 
Now
\[
\| f_h (Y(\floor{t})) - f(Y(\floor{t}))\| \leq C_a(R,h)
\]
and 
\[
\| f (Y(\floor{t})) - f(Y(t))\| \leq L(R)\|Y(\floor{t})-Y(t)\| 
\]
Thus 
\[ \mathbb{E} \sup_{[0,\tau_{R,h}]} \|\mathrm{Res}_1 (t)\|^p \leq C_p \left[C_a(R,h)+L(R)h^{1/2} B(R)\right]^p
\]
 The bound for $\mathrm{Res}_2 $ follows in a similar way.
  
\end{proof}

\begin{proof}[Proof of Lemma~\ref{lem:emom}]
Recall for the error
\[dE = [f(x)-f(x+E)]dt +  [g(x)-g(x+E)]dW + \mathrm{Res}_1 dt + \mathrm{Res}_2 dW
\]
Thus, using Ito-formula 
we obtain 
for some constant $C_\epsilon$  depending only on  $\epsilon$
\begin{align}
 d\|E\|^2 &= 2 \langle E, dE \rangle + \langle dE, dE \rangle\nonumber\\
 &= 2 \langle E, [f(x)-f(x+E)] + \mathrm{Res}_1 \rangle dt + 2 \langle E , [g(x)-g(x+E) + \mathrm{Res}_2] dW \rangle \nonumber\\
 &\quad + \|[g(x)-g(x+E)] +\mathrm{Res}_2 \|_{\HS}^2  dt \nonumber\\ 
 &\leq
 2 \langle E, [f(x)-f(x+E)] \rangle dt + (1+\epsilon)\|g(x)-g(x+E)\|_{\HS}^2 dt + \epsilon \|E\|^2  dt \nonumber\\
 &\quad + \left[ C_\epsilon \|\mathrm{Res}_2 \|_{\HS}^2  +  C_\epsilon \| \mathrm{Res}_1\|^2 \right] dt + 2 \langle E , [g(x)-g(x+E) + \mathrm{Res}_2 ]dW \rangle \nonumber\\
 &\leq 
  \left[ \delta(R) \|E\|^2 +   C_\epsilon \|\mathrm{Res}_2 \|_{\HS}^2  +  C_\epsilon\| \mathrm{Res}_1\|^2 \right] dt 
  \nonumber\\&\quad + 2 \langle E , [g(x)-g(x+E)+ \mathrm{Res}_2] dW \rangle \label{e:bounde}
\end{align}

This yields from Lemma \ref{lem:Res} using the martingale property of the stopped integrals 
\begin{align*}
\mathbb{E}\|E(t \wedge\tau_{R,h})\|^2 
&\leq  \|E(0)\|^2 + \delta(R) \mathbb{E}\int_0^{t\wedge \tau_{R,h}}\|E\|^2 dt 
+   C_\epsilon T K(R,h)^2 
  \\&\leq  \|E(0)\|^2 + \delta(R) \mathbb{E}\int_0^t\|E(s \wedge\tau_{R,h}) \|^2 dt 
+ C  K(R,h)^2
\end{align*}
where the constant depends on $T$ and the choice of $\epsilon$. Assume first that $\delta(R)>0$. 
Using Gronwall's lemma and $E(0)=0$ we obtain the bound 
\[ \mathbb{E}\|E(t \wedge\tau_{R,h})\|^2 
\leq  
C \int_0^t e^{ \delta(R) s } ds  K(R,h)^2
\]

Assume now that $\delta(R)\leq0$.
This yields from \eqref{e:bounde} using martingale property of the stopped integrals 
\[
\mathbb{E}\|E(t \wedge\tau_{R,h})\|^2 + |\delta(R)| \mathbb{E}\int_0^{t\wedge \tau_{R,h}}\|E\|^2 dt 
\leq  
  C K(R,h)^2.
\]
\end{proof}

\begin{proof}[Proof of Lemma~\ref{lem:uniform_mom}]
Recall from \eqref{e:bounde} for $t\leq \tau_{R,h}$
\begin{align*}
 \|E(t)\|^2
 &\leq
   \int_0^t  \left[\delta(R) \|E\|^2 +   C  \|\mathrm{Res}_2 \|_{\HS}^2+  C  \| \mathrm{Res}_1\|^2 \right] dt 
  \\& + 2   \int_0^t \langle E , [g(x)-g(x+E) + \mathrm{Res}_2 ] dW \rangle
\end{align*}
Thus, using Burkholder-Davis-Gundy (recall $\tau_{R,h}\in[0,T]$) assuming $\delta(R)>0$ 
\begin{align*}
 \mathbb{E} \sup_{[0,\tau_{R,h}]}\|E\|^2
 &\leq 
    \mathbb{E}  \int_0^{\tau_{R,h}}  \left[\delta(R) \|E\|^2 +   C  \|\mathrm{Res}_2 \|_{\HS}^2+  C  \| \mathrm{Res}_1\|^2 \right] ds
  \\&\quad + 2   \mathbb{E} \Big(  \int_0^{\tau_{R,h}} \left[L(R)^2\|E\|^4  +  \|E\|^2\|\mathrm{Res}_2\|_{\HS}^2 \right] dt  \Big)^{1/2}
    \\ &\leq 
    \delta(R) \int_0^{T}   \mathbb{E}  \|E(s\wedge\tau_{R,h})\|^2 ds \\    
    &\quad+   C  \mathbb{E} \sup_{[0,\tau_{R,h}]} \|\mathrm{Res}_2 \|_{\HS}^2
    +  C  \mathbb{E} \sup_{[0,\tau_{R,h}]} \| \mathrm{Res}_1\|^2    
  \\& + C  \Big(  (L(R)^2+1) \int_0^T  \mathbb{E}\|E(s\wedge\tau_{R,h})\|^4 ds    
  + \mathbb{E} \sup_{[0,\tau_{R,h}]} \|\mathrm{Res}_2\|_{\HS}^4  \Big)^{1/2}
\end{align*}
Using Lemma \ref{lem:Res} we obtain 
\begin{align*}
  \mathbb{E} \sup_{[0,\tau_{R,h}]}\|E\|^2 
  &\leq   \delta(R) \int_0^{T}   \mathbb{E}  \|E(s\wedge\tau_{R,h})\|^2 ds \\
  &\quad + C   (L(R)^2+1) \Big( \int_0^T  \mathbb{E}\|E(s\wedge\tau_{R,h})\|^4 ds \Big)^{1/2}
   + C K(R,h)^2.
\end{align*}
Moreover in the case $\delta(R)\leq0$ we have similarly 
\[
  \mathbb{E} \sup_{[0,\tau_{R,h}]}\|E\|^2 
  \leq  C   (L(R)^2+1) \Big(\int_0^T  \mathbb{E}\|E(s\wedge\tau_{R,h})\|^4 ds \Big)^{1/2}
  + C K(R,h)^2.
\] 
We obtain the assertion by using Lemma \ref{lem:emom} and Lemma \ref{lem:emom4}.

\end{proof}

\section{Proofs of Section~\ref{sec:general_EKI}}\label{app:general_EKI}

\begin{lemma}[An interpolation lemma for second moments]\label{lem:interpolation_quad}
Let $u(t) = u_0 + t\cdot f(u_0) + g(u_0)W_t$ with $u_0$, $W_t$ independent. Assume further that $\E \|u_0\|^2 < C$ and $\E \|u(1)\|^2 < C$, then $\E \|u(t)\|^2 < C$ for all $t\in [0,1]$.
\end{lemma}
\begin{proof}
Note first that by independence, $\E[h(u_0)W_t] = 0$ and 
\[
\E[h(u_0)^2W_t^2] = \E[h(u_0)]^2 \E[W_t]^2 
= \E[h(u_0)]^2t
\]
for (suitably integrable) functions $h$.  We compute first 
\begin{align*}
\E [u(1)]^2 &= \E[u_0 + f(u_0) + g(u_0)W_1]^2 = \E[u_0 + f(u_0)]^2 + 2\cdot 0 + \E[g(u_0)]^2\cdot 1
\end{align*}

Thus,
\begin{align*}
\E [u_0 + f(u_0)t + g(u_0)W_t]^2 &=  \E[u_0 + f(u_0)t]^2 
+ 2\cdot 0 + \E [g(u_0)W_t]^2 \\
&= \E[(1-t)u_0 + t(u_0+g(u_0))]^2 + \E[g(u_0)]^2 t.
\end{align*}
Now we note that $((1-t)a+tb)^2 \leq (1-t)a^2 + t(a+b)^2$ by Jensen's inequality 
\begin{align*}
\E [u_0 + f(u_0)t + g(u_0)W_t]^2 &\leq (1-t)\E u_0^2 + t\E[u_0+f(u_0)]^2 + t\E[g(u_0)]^2\\
&= (1-t)\E u_0^2 + t\E[u(1)]^2
\end{align*}
from which the statement follows.

\end{proof}
We will need the following fundamental lemmata.

\begin{lemma}
Let $W\sim N(0,\sigma^2)$ be a centered Gaussian random variable. Then $\E |W|^p = C_p\cdot \left(\E|W|^2\right)^{\frac{p}{2}}$.
\begin{proof}Substituting $x = y\sigma$, we can compute
\begin{align*}
\E|W|^p &= \frac{1}{\sqrt{2\pi}\sigma}\int_{-\infty}^\infty |x|^p e^{-\frac{x^2}{2\sigma^2}}\d x = \sigma^p\frac{1}{\sqrt{2\pi}}\int_{-\infty}^\infty |y|^p e^{-\frac{y^2}{2}}\d y = \left(\E|W|^2\right)^{\frac{p}{2}} \cdot C_p.
\end{align*}
\end{proof}
\end{lemma}
For non-centered Gaussian random variables we can show
\begin{lemma}\label{lem:aux_gauss}
Let $Z\sim N(a,\sigma^2)$ be a Gaussian random variable. Then there is a constant $C_p > 0$ such that 
\[ \left(\E|Z|^2\right)^\frac{1}{2} \leq C_p \left(\E|Z|^p\right)^\frac{1}{p}.\]
\begin{proof}
We can assume $\sigma = 1$ by rescaling and set $Z = a + W$ with $W\sim N(0,1)$. Now we consider
\begin{align*}
\frac{\left(\E|Z|^p\right)^\frac{1}{p}}{\left(\E|Z|^2\right)^\frac{1}{2}} &=\frac{\left(\E|a+W|^p\right)^\frac{1}{p}}{\left(\E|a+W|^2\right)^\frac{1}{2}} =: f_p(a)
\end{align*}
as a function of $a$. If we can show that $\inf_a f_p(a) > 0$, then the statement follows with $C_p = \left(\inf_a f_p(a)\right)^{-1}$. Evidently $f_p(a) > 0$ for all $a\in \R$, also $f_p(-a) = f_p(a)$, and $f_p$ is a continuous map. Thus, if we can show that $\lim_{a\to\infty} f_p(a) > 0$, then $\inf_a f_p(a) > 0$. We start by noting that $\E|a+W|^2 = a^2+1$. Then
\begin{align*}
\lim_{a\to\infty} (f_p(a))^p &= \lim_{a\to\infty} \E\left|\frac{a+W}{\sqrt{1+a^2}}\right|^p =  \lim_{a\to\infty} \E\left|\frac{a+W}{a}\right|^p \left|\frac{a}{\sqrt{1+a^2}}\right|^p \\
&= \lim_{a\to \infty} \E\left|1 + a^{-1}W\right|^p \geq \lim_{a\to \infty} (1-\epsilon)^{-p}\cdot \mathbb P\left(|1+a^{-1}W|\geq 1-\epsilon\right)\\
&= (1-\epsilon)^{-p},
\end{align*}
where we used Chebyshev's inequality and
\begin{align*}
\mathbb P\left(|1+a^{-1}W|\geq 1-\epsilon\right) \geq \mathbb P\left(1+a^{-1}W\geq 1- \epsilon\right) \geq \mathbb P\left(W\geq -a\epsilon\right).
\end{align*}
As $\lim_{a\to\infty} (f_p(a))^p \geq \sup_{\epsilon \in (0,1)} (1-\epsilon)^{-p} = 1 > 0$, we have shown the statement.
\end{proof}
\end{lemma}

\begin{lemma} \label{lem:aux_gauss2}
Let $W\sim N(0,1)$. Then for $a,b\in\R$, $p\in(0,2)$ and $t\in(0,1)$, we have
\[ \E|a+tb+\sqrt{t}cW|^p \leq C_p\cdot\left[|a|^p + \E|a+b+cW|^p\right].\]
\begin{proof}
We note that for for random variables $\E|X|^p\leq \left(\E|X|^2\right)^{\frac{p}{2}}$ by Hölder's inequality. Also, $|w+z|^{\frac{p}{2}} \leq |w|^{\frac{p}{2}} + |z|^{\frac{p}{2}}$. Thus, using lemma \ref{lem:interpolation_quad},
\begin{align*}
\E|a+tb+\sqrt{t}cW|^p &\leq \left(\E|a+tb+\sqrt{t}cW|^2 \right)^{\frac{p}{2}}\\
&\leq \left( a^2 + \E |a+b+cW|^2\right)^{\frac{p}{2}}\\
&\leq |a|^p + \left( \E |a+b+cW|^2\right)^{\frac{p}{2}}\\
&\leq |a|^p + C_p \E |a+b+cW|^p
\end{align*}
with the last step being due to lemma \ref{lem:aux_gauss}.
\end{proof} 
\end{lemma}

 \begin{proof}[Proof of lemma \ref{lem:interpolation}]
 We first consider the case where all stochastic processes and random variables involved are one-dimensional. Then the statement is a consequence of lemma \ref{lem:aux_gauss2} after seeing that
 \begin{align*}
  \E|u_0 + tf(u_0) + g(u_0)W_t|^p = \E\left[\E\left[|u_0 + tf(u_0) + g(u_0)W_t|^p | \mathcal F_0\right]\right]
 \end{align*}
 and identifying $a = u_0$, $b = f(u_0)$, $\sqrt{t}cW = g(u_0)W_t$ (where we can use $\sqrt{t}W = W_t$ in distribution for $W\sim N(0,1)$). The higher-dimensional case then follows from the one-dimensional considerations by seeing that for a random vector $Z$, 
 \begin{align*}
 \E \|Z\|^p &= \E\left(\sum_{i=1}^d |z_i|^2\right)^\frac{p}{2} \simeq \left(\sum_{i=1}^d \E|z_i|^p\right),\\ 
 \left(\E \|Z\|^2\right)^\frac{p}{2} &= \left(\E\sum_{i=1}^d |z_i|^2\right)^\frac{p}{2} \simeq \sum_{i=1}^d \E\left( |z_i|^2\right)^\frac{p}{2},
 \end{align*}
 where $x\simeq y$ means that there exist constants $a,A > 0$ such that $a x \leq y \leq A x$.
 \end{proof}

\begin{proof}[Proof of Theorem~\ref{thm:EKI_main1}]

Recall that by Theorem \ref{thm:EKI_general1},
 we just have to verify
that
there exists $\varphi$ (monotone growing) such that 
\[
\mathbb{E}\varphi(\|u(\tau_{R,h})\|) \leq C.
\]
We first introduce the short-hand notation 
\[
\mathcal{F}(u)= C^{up}(u)\Gamma^{-1/2}
\]
and rewrite 
\[
du^{(j)}=-\mathcal{F}(u) \Gamma^{-1/2}(G(u^{(j)})-y) dt + \mathcal{F}(u)dW.
\]
Denote by $\overline{u}$, $\overline{W}$ and  $\overline{G}$
the mean values of $u^{(j)}$, $W^{(j)}$, and $G(u^{(j)})$
with respect to $j$.
Thus, 
\[
d\overline{u}=-\mathcal{F}(u) \Gamma^{-1/2}(\overline{G}-y) dt + \mathcal{F}(u)d\overline{W}
\]
and
\[
d(u^{(j)}-\overline{u}) 
= -\mathcal{F}(u) \Gamma^{-1/2}(G(u^{(j)}) - \overline{G}) dt + \mathcal{F}(u)d(W^{(j)}-\overline{W}).
\]
By It\^o-formula we obtain
\begin{align*}
 d\|u^{(j)}-\overline{u}\|^2 
 &= 2 \langle u^{(j)}-\overline{u} , d(u^{(j)}-\overline{u}) \rangle
 +\langle d(u^{(j)}-\overline{u}) , d(u^{(j)}-\overline{u}) \rangle\\
 &= -2  \langle u^{(j)}-\overline{u} , \mathcal{F}(u) \Gamma^{-1/2}(G(u^{(j)}) - \overline{G}) \rangle dt
\\ &\quad
 +2  \langle u^{(j)}-\overline{u} , \mathcal{F}(u)d(W^{(j)}-\overline{W}) \rangle \\&\quad
 + \langle \mathcal{F}(u)d(W^{(j)}-\overline{W}) ,  \mathcal{F}(u)d(W^{(j)}-\overline{W})\rangle
\end{align*}
Now we use that
\begin{align*}
 \frac1J\sum_j\langle u^{(j)}-\overline{u} ,& \mathcal{F}(u) \Gamma^{-1/2}(G(u^{(j)}) - \overline{G}) \rangle
 \\ &={\frac1J\sum_j\trace(\mathcal{F}(u) \Gamma^{-1/2}(G(u^{(j)}) - \overline{G})(u^{(j)}-\overline{u})^\top)}\\
 &= {\trace(F(u)F(u)^\top)}= \| \mathcal{F}(u) \|^2_{\HS} 
\end{align*}
and 
\[
 \langle \mathcal{F}(u)d(W^{(j)}-\overline{W}) ,  \mathcal{F}(u)d(W^{(j)}-\overline{W})\rangle = 2(1-\frac1J) \| \mathcal{F}(u) \|^2_{\HS} dt
\]
to obtain 
\begin{equation*}
 d \frac1J\sum_j \|u^{(j)}-\overline{u}\|^2 
 =-\frac2J \| \mathcal{F}(u) \|^2_{\HS} dt
 +2  \frac1J\sum_j\langle u^{(j)}-\overline{u} , \mathcal{F}(u)d(W^{(j)}-\overline{W}) \rangle.
\end{equation*}
The martingale term vanishes in expectation if we intergrate up to stopping times such that $u$ remains bounded. Thus, we obtain the first main result of this proof.

For all $t\in[0,T]$, $R>1$ and $h\in(0,1)$ we have 
\begin{align}\label{e:mainbound}
{\mathbb{E} \frac1J\sum_j \|u^{(j)}-\overline{u}\|^2(t\wedge \tau_{R,h}) 
+ \frac2J \int_0^{t\wedge \tau_{R,h}} \| \mathcal{F}(u) \|^2_{\HS} ds}
\leq 
\mathbb{E} \frac1J\sum_j \|u^{(j)}-\overline{u}\|^2(0)
\end{align}
In this result we did not use any particular property of $G$. 
It remains to bound $\overline{u}$ now, which is the crucial point that leads to restrictions.
First by Ito-formula
\begin{align*}
 d\|\overline{u}\|^2 
 &= 2 \langle \overline{u} , d\overline{u} \rangle
 +\langle d\overline{u} , d\overline{u} \rangle\\
 &= 2  \langle \overline{u} , \mathcal{F}(u) \Gamma^{-1/2}( \overline{G}-y) \rangle dt
 +  (2-\frac2J)\| \mathcal{F}(u) \|^2_{\HS} dt
 \rangle \\&\quad
 +  \langle \overline{u} , \mathcal{F}(u)d(W^{(j)}-\overline{W})
\end{align*}
Here, we cannot use cancellations as in the step before. Therefore, we define for $z\geq0$ the function 
\[
\varphi(z)= \ln(1+z) 
\quad\mathrm{with}\quad 
z\varphi'(z) \leq 1  
\quad\mathrm{and}\quad 
z^2\varphi''(z) \leq 1  .
\]
Again using Ito-formula, we have 
\begin{align}
 d\varphi(\|\overline{u}\|^2) 
 &= \varphi'(\|\overline{u}\|^2) d \|\overline{u}\|^2 
 +  \varphi''(\|\overline{u}\|^2) d \|\overline{u}\|^2d \|\overline{u}\|^2 \nonumber\\
 &=
 2\varphi'(\|\overline{u}\|^2)\langle \overline{u} , \mathcal{F}(u) \Gamma^{-1/2} \overline{G} \rangle dt 
 \label{e:t1}\\&\quad
  -2\varphi'(\|\overline{u}\|^2)\langle \overline{u} , \mathcal{F}(u) \Gamma^{-1/2}y \rangle dt 
  \label{e:t2}\\&\quad
 +  (2-\frac2J)\varphi'(\|\overline{u}\|^2)\| \mathcal{F}(u) \|^2_{\HS} dt
  \label{e:t3}\\&\quad +
 \varphi'(\|\overline{u}\|^2)\langle \overline{u} , \mathcal{F}(u)d(W^{(j)}-\overline{W})
  \label{e:t4}\\&\quad + (2-\frac2J) \varphi''(\|\overline{u}\|^2)
 \langle \overline{u} , \mathcal{F}(u)\mathcal{F}(u)^T\overline{u}\rangle.
  \label{e:t5}
\end{align}
Now we have to bound all terms separately.
The martingale term in (\ref{e:t4}) vanishes in expectation,
if we integrate up to $t\wedge \tau_{R,h}$.
Now
\[
(\ref{e:t3}) \leq  
C \| \mathcal{F}(u) \|^2_{\HS} dt
\]
which is integrated up to $t\wedge \tau_{R,h}$ in expectation bounded by (\ref{e:mainbound}).
We bound similarly
\[
(\ref{e:t5}) \leq  
 (2-\frac2J) \varphi''(\|\overline{u}\|^2) \|\overline{u}\|^2
 \| \mathcal{F}(u) \|^2_{\HS} dt  \leq 
 C \| \mathcal{F}(u) \|^2_{\HS} dt
\]
and
\[
(\ref{e:t2}) \leq  
 2 \varphi''(\|\overline{u}\|^2) \|\overline{u}\| \|y\|\|\mathcal{F}(u)\|_{\HS} \|\Gamma^{-1/2}\|_{\HS}dt  \leq 
 C (1+\| \mathcal{F}(u) \|^2_{\HS}) dt.
\]
The crucial term is (\ref{e:t1}). Here, we have 
\[
(\ref{e:t1}) 
\leq  2 \frac{\|\overline{u}\|\|\overline{G}\|}{1+\|\overline{u}\|^2} 
\|\mathcal{F}(u)\|_{\HS} \|\Gamma^{-1/2}\|_{\HS}
 \leq 
  \frac{\|\overline{G}\|^2}{1+\|\overline{u}\|^2} 
  + C\|\mathcal{F}(u)\|_{\HS}^2
\]
Now we need to use that $G$ is Lipschitz to obtain
\[\|\overline{G}\| 
\leq C(\frac1J\sum_j\|u^{(j)}\|+1)
\leq C(\frac1J\sum_j\|u^{(j)}-\overline{u}\|+ \|\overline{u}\| + 1)
\]
which implies (for constants depending on $J$)
\[
(\ref{e:t1}) 
\leq  C ( 1+\frac1J\sum_j\|u^{(j)}-\overline{u}\|^2 + \|\mathcal{F}(u)\|_{\HS}^2.
\]
Integrating from $0$ to $t\wedge \tau_{R,h}$ we  
finally obtain 
together with the bound from (\ref{e:mainbound})
for all $R>1$ and $h\in(0,1)$ that 
\[
\mathbb{E} \varphi(\|\overline{u}(t\wedge \tau_{R,h})\|^2) \leq C.
\]
But as $\varphi$ satisfies $\varphi(x+y)\leq \varphi(x) + y$ we obtain, again using  (\ref{e:mainbound})
\[
\mathbb{E} \varphi(\frac1J\sum_j\|u^{(j)}(t\wedge \tau_{R,h})\|^2) \leq C.
\]
which finishes the proof.
\end{proof}

\section{Proofs of Section~\ref{sec:linear_EKI}}\label{app:linear_EKI}

We are going to prove the following usefull auxilary result which we are going to apply at several points.
\begin{lemma}\label{lem:nonneg_innerprod}
Let $S$ be a symmetric and nonnegative $d\times d$-matrix, then for all choices of vectors $(z^{(k)})_{k=1,\dots,J}$ in $\R^d$ we have 
\begin{equation*}
\sum\limits_{k, l=1}^J\langle z^{(k)},z^{(l)}\rangle\langle z^{(k)},Sz^{(l)}\rangle\ge0.
\end{equation*}
\end{lemma}
\begin{proof}
Let $(v^{(m)})_{m=1,\dots,d}$ be an orthonormal basis of eigenvectors such that $Sv^{(m)}=\lambda_mv^{(m)}$ with $\lambda_m\ge0$. Then $z^{(l)}=\sum\limits_{m=1}^dz_m^{(l)}v^{(m)}$ and thus
\begin{align*}
\sum\limits_{k, l=1}^J\langle z^{(k)},z^{(l)}\rangle\langle z^{(k)},Sz^{(l)}\rangle&=\sum\limits_{k, l=1}^J\sum\limits_{m, n=1}^dz_n{(k)}z_n^{(l)}z_m^{(k)}z_m^{(l)}\lambda_m\\ &= \sum\limits_{n, m=1}^d\lambda_m(\sum\limits_{k=1}^Jz_n^{(k)}z_m^{(k)})^2\ge0.
\end{align*}
\end{proof}

\begin{proof}[Proof of Lemma~\ref{lem:inner_prod}]
First we define the operator
\begin{equation*}
M^\varepsilon(u_n) := (hB(C(u_n)+\varepsilon I_p)B^T+I_K)^{-1},
\end{equation*}
for which it holds true, that 
\[
\lim\limits_{\varepsilon\to0}\ M^\varepsilon(u_n) = M(u_n), \] 
since the mapping $\Sigma\mapsto\Sigma^{-1}$ is continuous over the set of invertible matrices. By
\begin{align*}
C(u_n) B^T M(u_n)\tilde y = \frac1J\sum_{k=1}^J \langle B(u_n^{(k)}-\bar u_n),M(u_n)\tilde y\rangle (u_n^{(k)}-\bar u_n),
\end{align*}
it is sufficient to prove 
\[
\langle B(u_n^{(k)}-\bar u_n),M(u_n)\tilde y\rangle =0.
\]
We introduce 
\[
C^\varepsilon(u_n) := C(u_n)+\varepsilon I_p
\] 
and apply the Woodbury-matrix identity
\begin{align*}
  &\langle B(u_n^{(k)}-\bar u_n),M^\varepsilon(u_n)\tilde y\rangle\\ &= \langle B(u_n^{(k)}-\bar u_n),\left[I_K^{-1}-h I_K^{-1}B((C^\varepsilon(u_n))^{-1}+h B^T I_K^{-1}B)^{-1}B^TI_K^{-1})\right]\tilde y\rangle\\ 
 				&= \langle B(u_n^{(k)}-\bar u_n),\tilde y\rangle - \langle B(u_n^{(k)}-\bar u_n),hB((C^\varepsilon(u_n))^{-1}+hB^TB)^{-1}B^T)\tilde y\rangle\\
 				& = 0 - \langle hB\left[((C^\varepsilon(u_n))^{-1}+hB^TB)^{-1}\right]^TB^T B(u_n^{(k)}-\bar u_n),\tilde y\rangle\\
 				& = 0,
\end{align*}
where we have used that $\tilde y\in\range(B)^\perp$. We conclude with
\begin{equation*}
\langle B(u_n^{(k)}-\bar u_n),M(u_n) \tilde y\rangle = \lim\limits_{\varepsilon\to0}\ \langle B(u_n^{(k)}-\bar u_n),M^\varepsilon(u_n) \tilde y\rangle = 0.
\end{equation*}

\end{proof}

\begin{proof}[Proof of Lemma~\ref{lem:deviationsdecrease}]
We can derive the evolution of the euclidean norm by
  \begin{align*}
  \|e_{n+1}^{(j)}\|^2  &= \|e_n^{(j)}\|^2 {- 2h\langle e_n^{(j)}, C(u_n) B^TM(u_n)B e_n^{(j)}\rangle} \\&\qquad + 2\sqrt h \langle e_n^{(j)}, C(u_n)B^TM(u_n)(W_{n+1}^{(j)} - \bar W_{n+1})\rangle \\
  &\qquad -2h^{3/2}\langle C(u_n) B^T M(u_n) Be_n^{(j)}, C(u_n)B^TM(u_n)(W_{n+1}^{(j)} - \bar W_{n+1})\rangle \\ 
  &\qquad + h^2 \|C(u_n) B^T M(u_n) Be_n^{(j)}\|^2 \\
  &\qquad + h \|C(u_n)B^TM(u_n)(W_{n+1}^{(j)} - \bar W_{n+1})\|^2.
  \end{align*}

  We first write, after plugging in the definition of $C(u_n)$ and inserting 
  \[
  M(u_n) (hBC(u_n)B^T + I)\ = I,
  \]
  also abbreviating $M = M(u_n)$ and $C = C(u_n)$
  \begin{align*}
-2h\langle e_n^{(j)}, C(u_n) B^TMB e_n^{(j)}\rangle &= -2h \cdot \frac1J\sum_j \langle e_n^{(j)}, C B^TMB e_n^{(j)}\rangle\\ 
&= -2h \frac1J\sum_j \langle  e_n^{(j)}, C B^T\cdot M[hBCB^T]\cdot MB e_n^{(j)}\rangle \\
  &\quad -2h \frac1J\sum_j \langle  e_n^{(j)}, C B^T\cdot M\cdot MB e_n^{(j)}\rangle\\
  &= -2h^2\frac1J\sum_j \langle B^TMBCe_n^{(j)}, CB^TMBe_n^{(j)}\rangle \\
  &{\quad -2h\frac1J\sum_j \langle MBCe_n^{(j)}, MBe_n^{(j)}\rangle}.
  \end{align*}
  
Defining $Z = B^TMB$ (this proof works for any self-adjoint matrix) it is easy to verify
  \begin{align*}\frac1J\sum_j \langle {{ZCe^{(j)}}}, {CZ e^{(j)}}\rangle &= \frac1J\sum_l \|CZe^{(l)}\|^2
  \end{align*}
  and we can continue to write 
  \begin{align*}
  -2h \cdot \frac1J\sum_j \langle e_n^{(j)}, C B^TMB e_n^{(j)}\rangle &= -2h^2 \frac1J\sum_j \|CB^TMB e_n^{(j)}\|^2  \\
  &\quad - 2h\frac1J\sum_{j, k}\langle e_n^{(k)}, e_n^{(j)}\rangle\langle MBe_n^{(k)}, MBe_n^{(j)}\rangle,
  \end{align*}
  where we have used the definition of $C(u_n)$ for the second term.
  
   We define $S := M B$ and use $\E \langle a, W_i\rangle \langle b, W_j\rangle = \delta_{i,j}\langle a, b\rangle$ in order to derive
  \begin{align*}
     &\E\left[ \left\|CB^TM(W_{n+1}^{(j)}-\bar W_{n+1})\right\|^2\mid \cF_n\right] \\
     &=\E\left[\frac1{J^2}\sum_{k,l}\left\langle  e^{(k)} \langle e^{(k)}, S^T (W_{n+1}^{(j)}-\bar W_{n+1})\rangle, e^{(l)} \langle e^{(l)}, S^T(W_{n+1}^{(j)}-\bar W_{n+1})\rangle\right\rangle\mid \cF_n\right]\\
     &= \frac1J\sum_{k,l}\langle e^{(k)}, e^{(l)}\rangle \left[ \frac{J-1}{J^2}\langle Se^{(k)}, Se^{(l)}\rangle + \frac{(J-1)^2}{J^2}\langle Se^{(k)}, Se^{(l)}\rangle\right]\\
     &=\frac1J\sum_{l, k} \langle e^{(l)}, e^{(k)}\rangle \langle MB e^{(l)}, MB e^{(k)}\rangle \cdot \frac{J-1}{J}. 
\end{align*}

  We take the expectation up to step $n$ above to obtain 
  \begin{align*}\label{eq:deviation_decrease}
  \E \left[\frac1J\sum_j \|e_{n+1}^{(j)}\|^2 - \|e_n^{(j)}\|^2\mid \cF_n\right] &= -2h^2 \frac1J\sum_j \|CB^TMB e_n^{(j)}\|^2\nonumber \\
  &\quad- 2h\frac1J\sum_{j,k}\langle e_n^{(k)}, e_n^{(j)}\rangle\langle MBe_n^{(k)}, MBe_n^{(j)}\rangle\nonumber \\
  &\quad+ 0 + 0 + h^2\frac1J\sum_j \|CB^TMB e_n^{(j)}\|^2\nonumber \\
  &\quad+ \frac{J-1}{J^2}\cdot h\sum_{j,k}\langle e_n^{(k)}, e_n^{(j)}\rangle\langle  MBe_n^{(k)},  MBe_n^{(j)}\rangle\nonumber\\
  &= -h^2 \frac1J\sum_j \|CB^TMB e_n^{(j)}\|^2\nonumber \\
  &\quad- \frac{J+1}{J}h\|CB^\top M\|_{\HS}^2\nonumber\\
  &\leq 0
  \end{align*}
  where positivity of the last sum follows from lemma \ref{lem:nonneg_innerprod} by setting $S = B^TM^2B$. In particular, the process $(\frac1J\sum_{j=1}^J\|e_n^{(j)}\|)_{n\in\N}$ is a supermartingale, and the assertion follows.
\end{proof}

\begin{proof}[Proof of Proposition~\ref{prop:bound_imagespace}]
The update of the mapped residuals is given by
\begin{equation*}
B r_{n+1}^{(j)} = B r_n^{(j)} - hB C(u_n) B^TM(u_n)B r_n^{(j)} + \sqrt{h}B C(u_n) B^TM(u_n)W_{n+1}^{(j)}\,.
\end{equation*}
Using $M(u_n)(hBC(u_n)B^\top+I)=I$ and abbreviating again $M=M(u_n)$ and  $C=C(u_n)$, we obtain
\begin{align*}
\E[\|Br_{n+1}^{(j)}\|^2 &- \|Br_n^{(j)}\|^2\mid \cF_n]\\ &=-2h\langle Br_n^{(j)},BCB^TMBr_n^{(j)}\rangle  + 0 + 0 + h^2	\|BCB^TMBr_n^{(j)}\|^2 \\
&\quad + h\|BCB^TMW_{n+1}^{(j)}\|^2\\
&= -2h\langle Br_n^{(j)},M(hBCB^\top+I)BCB^TMBr_n^{(j)}\rangle \\
&\quad + h^2	\|BCB^TMBr_n^{(j)}\|^2 + h\E\left[\|BCB^TMW_{n+1}^{(j)}\|^2\mid\cF_n\right]\\
&= -2h^2\langle Br_n^{(j)},MBCB^\top BCB^TMBr_n^{(j)}\rangle -2h\langle Br_n^{(j)},MBCB^TMBr_n^{(j)}\rangle \\
&\quad + h^2	\|BCB^TMBr_n^{(j)}\|^2 + h\E\left[\|BCB^TMW_{n+1}^{(j)}\|^2\mid\cF_n\right]\\
&= -2h^2\langle BCB^\top MBr_n^{(j)}, BCB^TMBr_n^{(j)}\rangle\\ &\quad -2h\langle C^{1/2}B^\top MBr_n^{(j)},C^{1/2}B^TMBr_n^{(j)}\rangle \\
&\quad + h^2	\|BCB^TMBr_n^{(j)}\|^2 + h\E\left[\|BCB^TMW_{n+1}^{(j)}\|^2\mid\cF_n\right]\\
&= -2h^2\|BCB^\top MBr_n^{(j)}\|^2  -2h\|C^{1/2}B^\top MBr_n^{(j)}\|^2\\ &\quad + h^2	\|BCB^TMBr_n^{(j)}\|^2 + h\E\left[\|BCB^TMW_{n+1}^{(j)}\|^2\mid\cF_n\right]\,.
\end{align*}
We note that 
\[
\E\left[\|BCB^TMW_{n+1}^{(j)}\|^2\mid\cF_n\right] =  h\frac{1}{J^2}\sum_{l=1}^J\|C^{1/2}B^\top MBe_n^{(l)}\|^2.
\]
Similarly as in the proof of Lemma~\ref{lem:deviationsdecrease} we obtain
 \begin{align*}
\frac1J\sum\limits_{j=1}^J\E[\|Be_{n+1}^{(j)}\|^2 - \|Be_n^{(j)}\|^2\mid \cF_n]  &= -h^2 \frac1J\sum_{j=1}^J \|CB^TMB e_n^{(j)}\|^2 \\
  &\quad- h\frac{J+1}{J^2}\sum_{j=1}^J\|C^{1/2}B^\top MBe_n^{(j)}\|^2.
\end{align*}

We conclude with 
\begin{align}\label{eq:res_decrease}
 \E[\frac1J\sum_{j=1}^J(\|Br_{n+1}^{(j)}\|^2+\|Be_{n+1}^{(j)}\|^2) &- \frac1J\sum_{j=1}^J(\|Br_n^{(j)}\|^2+\|Be_{n}^{(j)}\|^2)\mid \cF_n]\nonumber\\&=-h^2\frac1J\sum_{j=1}^J\|BCB^\top MBr_n^{(j)}\|^2\nonumber\\ &\quad-2h\frac1J\sum_{j=1}^J\|C^{1/2}B^\top MBr_n^{(j)}\|^2\nonumber \\
 &\quad-h^2 \frac1J\sum_{j=1}^J \|CB^TMB e_n^{(j)}\|^2\nonumber\\
 &\quad-h\frac{1}{J^2}\sum_{j=1}^J\|C^{1/2}B^\top MBe_n^{(j)}\|^2\nonumber\\&\le 0.
 \end{align}

\end{proof}

\begin{proof}[Proof of Corollary~\ref{cor:sumbound}]
From the proof of Lemma~\ref{lem:deviationsdecrease} we know that
\begin{align*}
  0\le \frac1J\sum_{j=1}^J\E \left[\|e_n^{(j)}\|^2\right] = \E\left[\|e_0^{(j)}\|^2\right] &-\sum_{k=0}^{n-1} h^2 \frac1J\sum_{j=1}^J \E[\|CB^TMB e_k^{(j)}\|^2]\nonumber \\
  &- \sum_{k=0}^{n-1} \frac{J+1}{J}h\E[\|CB^\top M\|_{\HS}^2]\nonumber\\
  \end{align*}
  and it implies that for all $n\in\N$ we have that
  \begin{align*}
   \sum_{k=0}^{n-1} \frac{J+1}{J}h\|CB^\top M\|_{\HS}^2 \le \E\left[\|e_0^{(j)}\|^2\right].
  \end{align*}
  The other bound follow similarly by using the update formula
  \begin{align*}
 0\le \frac1J\sum_{j=1}^J\E[\|Br_{n}^{(j)}\|^2+\|Be_{n}^{(j)}\|^2\mid \cF_n] &=\frac1J\sum_{j=1}^J\E[\|Br_{0}^{(j)}\|^2+\|Be_{0}^{(j)}\|^2]\\ &\quad - \sum_{k=0}^{n-1} h^2\frac1J\sum_{j=1}^J\|BCB^\top MBr_k^{(j)}\|^2\\ &\quad-2\sum_{k=0}^{n-1}h\frac1J\sum_{j=1}^J\E[\|C^{1/2}B^\top MBr_k^{(j)}\|^2]\nonumber \\
 &\quad-\sum_{k=0}^{n-1}h^2 \frac1J\sum_{j=1}^J \E[\|CB^TMB e_k^{(j)}\|^2]\nonumber\\
 &\quad-\sum_{k=0}^{n-1}h\frac{1}{J^2}\sum_{j=1}^J\E[\|C^{1/2}B^\top MBe_k^{(j)}\|^2]\nonumber.
 \end{align*}
\end{proof}

\begin{proof}[Proof of Lemma~\ref{lem:secondmoments_imagespace}]

We first note that 
\[\E[\|r_n^{(j)}\|^2] = \E[\|Pr_n^{(j)}\|^2] + \E[\|(I-P)r_n^{(j)}\|^2]
\]
and we consider both terms separately. 

\textbf{Step 1 - Bounding $\E[\|Pr_n^{(j)}\|^2]$:}

We observe that 
\[\|Pr_n^{(j)}\|^2 = \|B^\top(BB^\top)^-Br_n^{(j)}\|^2 \le \|B^\top(BB^\top)^-\|_{\HS}^2 \|Br_n^{(j)}\|^2.
\]
Application of Proposition~\ref{prop:bound_imagespace} gives the uniform bound in $n$ and $h$, i.e.
\[
\|Pr_n^{(j)}\|^2\le c_1
\]
for some $c_1>0$ independent of $n$ and $h$.

\textbf{Step 2 - Bounding $\E[\|(I-P)r_n^{(j)}\|^2]$:}

For the update of $(I-P)r_n^{(j)}$ we have that 
\begin{align*}
 (I-P)r_{n+1}^{(j)} &= (I-P)r_n^{(j)} - h(I-P)C(u_n)B^\top M(u_n) Br_n^{(j)}\\ &\quad + \sqrt{h} (I-P)C(u_n)B^\top M(u_n)W_{n+1}^{(j)}\\
							&= (I-P)r_n^{(j)}\\ &\quad+\frac1J\sum_{k=1}^J \langle -hM(u_n)Br_n^{(j)}+\sqrt{h}M(u_n)W_{n+1}^{(j)},Be_n^{(k)}\rangle (I-P)e_n^{(k)}.
\end{align*}
Similarly, we have that
\begin{align*}
 (I-P)e_{n+1}^{(j)} &= (I-P)e_n^{(j)} \\ &\quad +\frac1J\sum_{k=1}^J \langle -hM(u_n)Be_n^{(j)}+\sqrt{h}M(u_n)W_{n+1}^{(j)},Be_n^{(k)}\rangle (I-P)e_n^{(k)}\\
							&= (I-P)e_{0}^{(j)}.
\end{align*}

Hence, we imply that $(I-P)e_n^{(k)} = 0$ for all $k$, i.e.~$e_n^{(k)}$ is in the range of $P$, and it follows that
\[ 
(I-P)r_{n+1}^{(j)} = (I-P)r_n^{(j)} = (I-P)r_0^{(j)}
\]
Finally, we conclude with
\[
\|(I-P)r_{n+1}^{(j)}\|^2 = \|(I-P)r_0^{(j)}\|^2\le c_2.
\]
\end{proof}

\begin{proof}[Proof of Lemma~\ref{lem:1_moments}]
Let $p=1$ and write
\begin{align*}
 \|r_{n+1}\|_{L_p}&:= \E[\| r_{n+1}^{(j)}\|^{p}]^{1/p}\\ &= \|r_n^{(j)} - h C(u_n)B^\top M(u_n)Br_n^{(j)} + \sqrt{h}C(u_n)B^{\top} M(u_n)W_{n+1}^{(j)}\|_{L_p}\\
&\le \|r_n^{(j)}\|_{L_1} + \|C(u_n)^{1/2}\|_{L_2} \|h C(u_n)^{1/2}B^\top M(u_n)Br_n^{(j)}\|_{L_2}\\ &\quad+ \|\sqrt{h}C(u_n)B^{\top} M(u_n)W_{n+1}^{(j)}\|_{L_1}.
\end{align*}

First, note that we can write
\[
(C(u_n))^{1/2} = (1/J\cdot (e_n^{(1)}, e_n^{(2)},\dots,e_n^{(J)})(e_n^{(1)}, e_n^{(2)},\dots,e_n^{(J)})^\top)^{1/2}
\]
and hence, it holds true that
\[ 
\|C(u_n)^{1/2}\|_{L_2} \le \left(\frac{1}{J} \sum_{j=1}^J \E\|e_n^{(j)}\|^2\right)^{1/2}\le \left(\frac{1}{J} \sum_{j=1}^J \E\|e_0^{(j)}\|^2\right)^{1/2}=: C_1. 
\]
Furthermore, we can bound
\begin{align*}
 \|r_{n}^{(j)}\|_{L_1} \le \|r_0^{(j)}\|_{L_1} &+ C_1 \sum_{k=0}^n \|h C(u_k)^{1/2}B^\top M(u_k)Br_k^{(j)}\|_{L_2}\\ & + \sum_{k=0}^n \|\sqrt{h}C(u_n)B^{\top} M(u_n)W_{n+1}^{(j)}\|_{L_1}.
 \end{align*}

From Corollaryn~\ref{cor:sumbound} we have that for all $n\ge1$
\[
2\sum_{k=0}^n  h\E[\| C(u_k)^{1/2}B^\top M(u_k)Br_k^{(j)}\|^2] \le \E[\frac1J\sum_{j=1}^J\|Br_0^{(j)}\|^2 + \|Be_0^{(j)}\|^2]
\]
and it follows by Jensen's inequality that
\begin{align*}
 \sum_{k=0}^n \|h C(u_k)^{1/2}B^\top M(u_k)Br_k^{(j)}\|_{L_2} &\le \sum_{k=0}^{N\cdot T} \|h C(u_k)^{1/2}B^\top M(u_k)Br_k^{(j)}\|_{L_2} \\
&\le T\cdot \left(\sum_{k=0}^{N\cdot T} \frac1T h\E[\| C(u_k)^{1/2}B^\top M(u_k)Br_k^{(j)}\|^2]\right)^{1/2}\\
&\le \sqrt{T}\cdot \E\left[1/2\frac1J\sum_{j=1}^J\|Br_0^{(j)}\|^2 + \|Be_0^{(j)}\|^2\right]^{1/2}
\end{align*}
providing a uniform bound in $h$ for all $n$.  Similarly, we obtain from Corollary~\ref{cor:sumbound} that
\begin{align*}
\frac{J+1}{J}\sum_{k=0}^n \E[\|\sqrt{h}C(u_k)B^{\top} M(u_k)W_{k+1}^{(j)}\|^2]&=\frac{J+1}{J}\sum_{k=0}^n \E[h\|C(u_k)B^{\top} M(u_k)\|_{\HS}^2]\\  &\le \E[\frac1J\sum_{j=1}^J \|e_0^{(j)}\|^2]
\end{align*}
and applying again Jensen's inequality gives
\begin{align*}
 \sum_{k=0}^n \|\sqrt{h}C(u_n)B^{\top} M(u_n)W_{n+1}^{(j)}\|_{L_1} &\le \sum_{k=0}^{N\cdot T} \E[\|\sqrt{h}C(u_n)B^{\top} M(u_n)W_{n+1}^{(j)}\|^2]^{1/2} \\
&\le T\cdot \left(\sum_{k=0}^{N\cdot T} \frac1T h\E[\| C(u_k)B^\top M(u_k)\|_{\HS}^2]\right)^{1/2}\\
&\le \sqrt{T}\cdot \left(\frac{J}{J+1}\frac1J\sum_{j=1}^J\E[\|e_0^{(j)}\|^2]\right)^{1/2}.
\end{align*}
We conclude the proof by
\begin{align*}
 \|r_n^{(j)}\|_{L_1} &\le \|r_0^{(j)}\|_{L_1}\\ &\quad + \sqrt{T}\cdot  \left(\frac{1}{J} \sum_{j=1}^J \E\|e_0^{(j)}\|^2\right)^{1/2}\cdot \left(1/2\frac1J\sum_{j=1}^J\E[\|Br_0^{(j)}\|^2 + \|Be_0^{(j)}\|^2]\right)^{1/2}\\
 &\quad+ \sqrt{T}\cdot \left(\frac{J}{J+1}\frac1J\sum_{j=1}^J\E[\|e_0^{(j)}\|^2]\right)^{1/2}.
\end{align*}
\end{proof}

\begin{proof}[Proof of Lemma~\ref{lem:moments_TEKI}]
We again decompose $\tilde y = \hat y+y^\prime$, where $\hat y\in \range( \tilde B)$ and $y^\prime\in \range(\tilde B)^\perp$. By Lemma~\ref{lem:inner_prod} there exists $\hat u$ (not necessarily unique),  such that we can write the update for $r_n^{(j)} = u_n^{(j)} - \hat u$ by
\[
r_{n+1}^{(j)} = r_n^{(j)} - hC(u_n)\tilde B^TM(u_n)\tilde Br_n^{(j)} + \sqrt{h}C(u_n)\tilde B^TM(u_n)W_{n+1}^{(j)}.
\]
By Proposition~\ref{prop:bound_imagespace} it follows that
\begin{align*}
  \sup_{n\in\{1,\dots,N\}}&\frac1J\sum_{j=1}^J\E[\|\tilde Br_{n+1}^{(j)}\|_{\R^K\times\cX}^2+\|\tilde Be_{n+1}^{(j)}\|_{\R^K\times\cX}^2]\\ &\le \E[\frac1J\sum_{j=1}^J(\|\tilde Br_n^{(j)}\|_{\R^K\times\cX}^2+\|\tilde Be_n^{(j)}\|_{\R^K\times\cX}^2)].
\end{align*}
The definition of $\tilde B$ implies that
\[
\|\tilde Br_{n+1}^{(j)}\|_{\R^K\times\cX}^2 = \|B(u_{n+1}^{(j)}-\hat u)\|_{\R^K}^2 + \|(u_{n+1}^{(j)}-\hat u)\|_{\cX}^2
\]
and hence, we conclude with
\[
\sup_{n\in\{1,\dots,N\}} \E[\|u_n^{(j)}\|^2]\le C
\]
for all $j\in\{1,\dots,J\}$ where $C>0$ is independent from $h$.
\end{proof}

\end{document}